\documentclass{amsart}
\usepackage{amssymb}
\usepackage{amsfonts}
\usepackage{amssymb}
\usepackage{amsmath}
\usepackage{amsthm}
\usepackage{enumerate}
\usepackage{graphicx}
\usepackage{xfrac}

\usepackage{amsthm,amssymb}

\renewcommand{\leq}{\leqslant}
\renewcommand{\geq}{\geqslant}

\def\Phii{\Phi^{\hspace{-0.5pt}\mbox{\tiny $\bot$}}}

\makeatletter
\def\subsection{\@startsection{subsection}{3}%
  \z@{.5\linespacing\@plus.7\linespacing}{.3\linespacing}%
  {\bfseries\centering}}
\makeatother

\makeatletter
\def\subsubsection{\@startsection{subsubsection}{3}%
  \z@{.5\linespacing\@plus.7\linespacing}{.3\linespacing}%
  {\centering}}
\makeatother

\theoremstyle{definition}

\newtheorem{theorem}{Theorem}[section]
\newtheorem{definition}[theorem]{Definition}
\newtheorem{lemma}[theorem]{Lemma}
\newtheorem{proposition}[theorem]{Proposition}
\newtheorem{example}[theorem]{Example}

\newcounter{claimcounter}
\numberwithin{claimcounter}{theorem}
\newenvironment{claim}{\stepcounter{claimcounter}{\noindent {\bf Claim \theclaimcounter.}}}{}
\newenvironment{claimproof}[1]{\noindent{{\em Proof.}}\space#1}{\hfill $\rule{0.35em}{0.35em}$}

\newcommand{\pureindep}[1][]{%
  \mathrel{
    \mathop{
      \vcenter{
        \hbox{\oalign{\noalign{\kern-.3ex}\hfil$\vert$\hfil\cr
              \noalign{\kern-.7ex}
              $\smile$\cr\noalign{\kern-.3ex}}}
      }
    }\displaylimits_{#1}
  }
}

\newcommand{\indep}[2]{%
  \mathrel{
    \mathop{
      \vcenter{
        \hbox{%
\oalign{
\noalign{\kern-.3ex}\hfil$\vert$\hfil\cr
              \noalign{\kern-.7ex}
              $\smile$\cr\noalign{\kern-.3ex}
}
}
      }
}^{\!\!\!\!\!#2}_{\!\!\hspace{-0.1em}#1}
  }
}

\newcommand{\displayindep}[2]{%
  \mathrel{
    \mathop{
      \vcenter{
        \hbox{%
\oalign{
\noalign{\kern-.3ex}\hfil$\vert$\hfil\cr
              \noalign{\kern-.7ex}
              $\smile$\cr\noalign{\kern-.3ex}
}
}
      }
}^{\!\!\hspace{-0.1em}#2}_{\!\!\hspace{-0.1em}#1}
  }
}

\newcommand{\clindep}[1][]{\indep{#1}{\mathsf{cl}}}
\newcommand{\forkindep}[1][]{\indep{#1}{\mathsf{f}}}

\newcommand{\displayclindep}[1][]{\displayindep{#1}{\mathsf{cl}}\!\hspace{-0.15em}}
\newcommand{\displayforkindep}[1][]{\displayindep{#1}{\mathsf{f}}\!\hspace{-0.15em}}

\newcommand{\typesp}[3]{S^{\mathcal{{#1}}}_{{#2}}({#3})}

\newcommand*\dep{{=\mkern-1.2mu}}
\newcommand{\Dep}[2]{\dep(\vec{{#1}}, \vec{{#2}})}

\newcommand*\bota{{\bot\mkern-1.2mu}}

\newcommand{\botc}[1]{~\bot_{#1}~}
\def\boto{\ \bot\ }

\title{Dependence Logic in Pregeometries and $\omega$-Stable Theories}
\thanks{Research of the first author was carried out while he was a Master of Logic student at the Institute for Logic, Language and Computation, University of Amsterdam, and also partially supported by
grant TM-13-8847 of CIMO. Research of the second author was partially supported by
grant 251557 of the Academy of Finland. The authors would like to thank John Baldwin, Miika Hannula, \AA sa Hirvonen, Tapani Hyttinen and Juha Kontinen for useful conversations related to this paper.}

\author{Gianluca Paolini}
\address{Department of Mathematics and Statistics,  University of Helsinki, Finland} 

\author{Jouko V\"a\"an\"anen}
\address{Department of Mathematics and Statistics, University of Helsinki, Finland  \and 
        Institute for Logic, Language and Computation, University of Amsterdam, The Netherlands}
    
\date{}
\begin{document}
\maketitle

\begin{abstract}
We present a framework for studying the concept of independence  in a general context covering database theory, algebra and model theory as special cases. We show that well-known axioms and rules of independence for making inferences concerning basic atomic independence statements are complete with respect to a variety of semantics. Our results show that the uses of  independence  concepts in as different areas as database theory, algebra and model theory, can be completely characterized by the same axioms. We also consider concepts related to independence, such as dependence.
\end{abstract}


\section{Introduction}
\def\la{\langle}
\def\ra{\rangle}
\def\A{\mathcal{A}}
\def\D{\mathcal{D}}
\def\I{\mathcal{I}}
\def\J{\mathcal{J}}
\def\V{\mathcal{V}}
\def\C{\mathbf{C}}
\def\P{\mathcal{P}}
\def\S{\mathcal{S}}
\def\rest{\restriction}
\def\bx{\textbf{\textit{x}}}
\def\by{\textbf{\textit{y}}}
\def\bz{\textbf{\textit{z}}}
\def\bu{\textbf{\textit{u}}}
\def\bv{\textbf{\textit{v}}}
\def\bw{\textbf{\textit{w}}}
\def\vx{\vec{x}}
\def\vy{\vec{y}}
\def\vz{\vec{z}}
\def\vu{\vec{u}}
\def\vv{\vec{v}}
\def\vw{\vec{w}}
The concepts of dependence and independence  are ubiquitous in science. They appear e.g. in biology, physics, economics, statistics, game theory, database theory,
 and  last but not least, in algebra. Dependence and independence concepts of algebra have  natural analogues in model theory, more exactly in geometric stability theory, arising from pregeometries and non-forking. In this paper we show that all these dependence and independence concepts have a common core that permits a complete axiomatization.  
 
For a succinct presentation of our results we introduce the following auxiliary  concept. 

\begin{definition}
 An {\em independence structure} is a pair $\S=(I,\bot)$ where $I$  is a non-empty set and $\bot$ is a  binary relation in the set of finite subsets of $I$, satisfying the following axioms (we use boldface symbols $\bx,\by$, etc for finite set, and  shorten $\bx\cup \by$ to $\bx\by$):
\begin{description}
\item[I1] $\bx\boto\emptyset$ 
\item[I2] If $\bx\boto \by$, then $\by\boto \bx$
\item[I3] If $\bx\boto \by\bz$, then $\bx\boto \by$
\item[I4] If $\bx\boto \by$ and $\bx\by\boto \bz$, then $\bx\boto \by\bz$. 
\item[I5] If $\bx\boto \bx$, then $\bx\boto \by$
\end{description}
\end{definition}

 If a pair $(\bx,\by)$ of subsets of $I$ satisfy $\bx\boto \by$, we say that $\bx$ and $\by$ are {\em independent} of each other.

Whitney \cite{whitney} introduced the concept of a {\em matroid} (or pregeometry) which is a stronger concept than that of  independence structure. Ours comes close, and takes inspiration from the axiomatization of independence in database theory \cite{KLV} and statistics \cite{geiger1991axioms}. Compared to the more widely known concept of matroid  our concept has the advantage that it covers more cases (database dependence, statistical independence) but, on the other hand, our concept lacks some fundamental properties of matroids, such as the uniqueness of the cardinality of maximal independent sets. 

The concept of {\em independence } is well-known in model theory  but is usually a ternary relation $\bx\ \bot_\by\ \bz$. The beauty of the {\em binary} independence relation $\bx\boto \by$ is that its axioms (I1)-(I5) are complete with respect to a variety of choices of semantics. It is our purpose in this paper to demonstrate exactly this. The same is not true of the more complex concept $\bx\ \bot_\by\ \bz$, for which there are natural axioms but no completeness is known, rather, there are negative results \cite{herrmann_undec,herrmann_undec_corrig}.

Some trivial examples of independence structures are:
\begin{itemize}
\item {\em Full independence structure}: $\bx\boto \by$ holds for all $\bx$ and $\by$.  
\item {\em Trivial independence structure}: $\bx\boto \by$ holds if and only if $\bx=\emptyset$ or $\by=\emptyset$.  
\item {\em Degenerate independence structure}: $\bx\boto \by$ holds if and only if $\bx\cap \by=\emptyset$.  
\end{itemize}

A special property of degenerate independence structures is the following {\em additivity} property:
$$(\bx\boto \bz\wedge \by\boto \bz)\Rightarrow \bx\by\boto \bz.$$

More elaborate independence structures are invariably obtained from some richer background structure. Here are two examples:

\begin{example}\label{teamexample}
Suppose $I$ is a set of variables  with possible values in a set $M$. 
A {\em team} with domain $I$ is any set $X$ of functions $s:I\to M$ \cite{MR2351449}. An individual $s\in X$ is called an {\em assignment}. 
For such $X$ we can define, following \cite{GV}, $$\bx\boto \by$$ to hold if $$\forall s,s'\in X\exists s''\in X(s''\rest \bx=s\rest \bx\mbox{ and }s''\rest \by=s'\rest \by).$$
Now $\S_X=(I,\bot)$ is an independence structure which we call a {\em team independence structure}.
\end{example}

We can also think of a team as a database with the set $I$ as the set of attributes. In that case assignments would be called {\em records} or {\em rows} of the database, and $\S_X$ a {\em database independence structure}.

\begin{example}
  Suppose $\I$ is a vector space with $I$ as the set of its vectors.
 Define for finite subsets $\bx$ and $\by$ of $I$:
 $\bx\boto \by$ if $\bx$ and $\by$ are linearly independent in the usual sense that the generated subspaces satisfy $\la \bx\ra\cap \la \by\ra = \{0\}$. Then $\S_\I=(I,\bot)$ is an independence structure which we call a {\em vector space independence structure}.
\end{example}
 A characteristic property of vector space independence structures is the basic fact that all maximal independent sets have the same cardinality. This is not true in team independence structures in general, as is easy to see. Vector space independence structures are special cases of independence structures arising from pregeometries (a.k.a. matroids), see Section~\ref{preg} below.

The families of independence structures arising from teams, databases, vector spaces or pregeometries are all examples of {\em classes}  of independence structures. Such classes $\C$ are the main topic of this paper. A special emphasis is on classes $\C$ of independence structures (or sets) arising from pregeometries and non-forking in models of $\omega$-stable theories, see Section~\ref{daiist} below. For further results in this direction we refer to \cite{pao} and \cite{hytpao}.

We set up a formal ``calculus" of independencies $\bx\boto \by$ in independence structures and call such calculi, when combined with a particular class $\C$ of independence structures, {\em atomic logics}. They are {\em atomic} because in this paper we do not consider rules and meanings concerning logical operations at all. In \cite{MR2351449} and \cite{GV} the full first-order logic based on atomic dependence or independence formulas is developed in the context of team independence structures. It turns out that in this case the full logic has the same expressive power as existential second-order logic \cite{kont}. 

The {\em independence calculus} is defined as follows: Suppose $V$ is a set of variables (i.e. variable symbols). We consider finite sequences $\vx,\vy,$ etc of variables. Concatenation of $\vx$ and $\vy$ is denoted $\vx\vy$. We use $V^{<\omega}$ to denote the set of all such sequences.  The empty sequence is denoted $\emptyset$. Expressions of the form $\vx\boto \vy$, where $\vx$ and $\vy$ are finite sequences of elements of $V$, are called {\em independence atoms} over $V$. Related to the axioms (I1)-(I5), we have the following monotone operator $\Phii$ on $V^{<\omega}$: 
\begin{eqnarray*}
\Phii(\Sigma)&=&\{\vx\boto\emptyset\}\cup\\
& &\{\vx\boto \vy : \vy\boto \vx\in \Sigma\}\cup\\
& &\{\vx\boto \vy : \vx\boto \vy\vz \in \Sigma\}\cup\\
& &\{\vx\boto \vy\vz : \vx\boto \vy \in \Sigma\mbox{ and } \vx\vy \boto \vz\in \Sigma\}\cup\\
& &\{\vx\boto \vy : \vx\boto \vx\in \Sigma\}\cup\\
& &\{\vz\boto \vy : \vx\boto \vy\in \Sigma\mbox{ and $\vz$ a permutation of $\vx$}\}.
\end{eqnarray*}
If $\Sigma$ is a set of independence atoms atoms over $V$, we use $$\la\Sigma\ra$$ to denote the closure of $\Sigma$ under the operator $\Phii$. Intuitively, $\la\Sigma\ra$ is the set of  atoms that follow from atoms in $\Sigma$ by means of the rules (I1)-(I5).

\begin{definition}
We say that an independence atom $\phi$ is {\em derivable} from a set $\Sigma$ of  independence atoms, $$\Sigma\vdash_{\mathcal{I}} \phi,$$ if $\phi\in \la\Sigma\ra$.
\end{definition}

 We can interpret the variables $v\in V$ in an independence structure $\S=(I,\bot_\S)$ with a mapping  $s:V\to I$ called assignment (into $\S$).  This induces a canonical mapping $\vx\mapsto s(\vx)$ of finite sequences $\vx$ of variables to finite subsets $$s(( x_0,\ldots,x_{n-1}))=\{s(x_0),\ldots,s(x_{n-1})\}$$ of $I$.  Then for any  finite sequences $\vx$ and $\vy$ of elements of $V$ and any assignment $s$ it is natural to define
$$\S\models_s \vx\boto \vy\mbox{ if and only if }s(\vx)\ \bot_\S\ s(\vy).$$ If this is the case, we say that $s$ {\em satisfies} the independence $\vx\boto \vy$ in $\S$. We use the same convention for sets $\Sigma$ of atoms.

\begin{definition}
 Suppose  $\C$ is a class of independence structures. We say that a set $\Sigma$ of  independence atoms over $V$ {\em logically implies} the independence atom $\phi$ over $V$ in independence structures in $\C$, $$\Sigma\models_\C \phi,$$
 if every assignment in any $\S\in\C$ which satisfies  $\Sigma$, also satisfies $\phi$.
\end{definition}

We have now introduced the key concepts behind the main question investigated in this paper:
\medskip

\noindent{\bf Completeness Question for $\C$:} Given a class $\C$ of independence structures, a set $\Sigma$ of independence atoms and an independence atom $\phi$, is it true that $$\Sigma\vdash_{\mathcal{I}}\phi\iff\Sigma\models_\C \phi?$$

Obviously the answer to the Completeness Question for $\C$ depends heavily on $\C$. A positive answer is a sign of the richness of $\C$, and if we choose $\C$ to be the class of {\em all} independence structures, then a positive answer follows trivially. So we are only interested in small restricted classes $\C$.  If two classes $\C$ and $\C'$ give both a positive answer to the Completeness Question, then $$\Sigma\models_\C \phi\iff\Sigma\models_{\C'} \phi,$$ which establishes an affinity between $\C$ and $\C'$: basic properties of independence are governed by the same axioms (I1)-(I5) in both classes.

 We show in this paper that the Completeness Question can be answered positively for all the naturally arising classes $\C$ in the contexts of
 \begin{itemize}
\item Team semantics i.e. databases (\cite{geiger1991axioms,galliani_and_vaananen}),

\item Pregeometries (Theorem~\ref{compl_pregeom_atom_indep_logic}),
\item $\omega$-stable first-order theories (Theorem~\ref{saind}),

\item Infinite vector spaces over a countable field (Section~\ref{vsacf}),

\item Algebraically closed fields of fixed characteristic (Section~\ref{vsacf}).

\end{itemize} 

\def\rr{\Rightarrow}

We study also some related variants of independence. A simplification of the concept of independence is the following: 

\begin{definition}\label{ainddef}
  An {\em absolute independence structure} is a pair $\S=(I,\A)$ where $I$  is a non-empty set and $\A$ is a   set of finite subsets of $I$, closed under subsets. 
\end{definition}
  
  A typical example is the set of  sets which are linearly independent  in a vector space.

{\em Dependence} is a concept which occurs at least as frequently in science as independence. Consider for example, linear dependence in vector spaces, or functional dependence in databases. We can give dependence a similar treatment as that we gave for independence:

\begin{definition}\label{depdef}
  A {\em dependence structure} is a pair $\D=(I,\rr)$ where $I$  is a non-empty set and $\rr$ is a  binary relation in the set of finite subsets of $I$, satisfying the following axioms:
\begin{description}
\item[D1] $\bx\rr \bx$. 
\item[D2] If $\bx\rr \by\bz$, then $\bx\bu\rr \by$.
\item[D3] If $\bx\rr \by$ and $\by\rr \bz$, then $\bx\rr \bz$.
\item[D4] If $\bx\rr \by$ and $\bx\rr \bz$, then $\bx\rr \by\bz$.
\end{description}
\end{definition}

Dependence structures model typically linear dependence in vector spaces and functional dependence in databases. As a degenerate case we have the dependence structure $(I,\rr)$ where $$\bx\rr\by\mbox{ if and only if } \by\subseteq\bx.$$ 

We can build a ``dependence calculus", derivability of {\em dependence atoms} $\dep(\vx,\vy)$ (with the meaning that the interpretation $\bx$ and $\by$ of $\vx$ and $\vy$, respectively, satisfy $\bx\rr\by$) from sets of such atoms, define truth of dependence atoms in classes  $\C$ of dependence structures, and ask the Completeness Question as we did for classes of independence structures. A classic result in this respect is the Completeness Theorem of Armstrong \cite{DBLP:conf/ifip/Armstrong74}. We give some new positive answers to the Completeness Question for dependence structures (Theorems \ref{compl_clo_op_dep_logic} and \ref{char_stab_dep_with_acl_dep}).

We also consider a version of {\em conditional} independence, which turns out to be more complicated than the original unconditional independence (I1)-(I5).
\begin{definition}
  A {\em conditional independence structure} is a pair $\S=(I,\bot)$ where $I$  is a non-empty set and $\bot$ is a ternary relation between finite subsets $\bx,\by,\bz$ of $I$, denoted $\bx\ \bot_\by\ \bz$, satisfying the following axioms:
\begin{description}
	\item[C1] $\bx \botc{\bx} \by$.
	\item[C2] If $\bx \botc{\bz} \by$, then $\by \botc{\bz} \bx$.
	\item[C3] If $\bx  \bx'  \botc{\bz} \by  \by'$, then $\bx \botc{\bz} \by$.
	\item[C4] If $\bx \botc{\bz} \by$, then $\bx  \bz  \botc{\bz} \by  \bz$.
	\item[C5] If $\bx \botc{\bz} \by$ and $\bu \botc{\bz, \bx} \by$, then $\bu \botc{\bz} \by$.
	\item[C6] If $\by \botc{\bz} \by$ and $\bz  \bx \botc{\by} \bu$, then $\bx \botc{\bz} \bu$.
	\item[C7] If $\bx \botc{\bz} \by$ and $\bx  \by \botc{\bz} \bu$, then $\bx \botc{\bz} \by  \bu$.
\end{description}
\end{definition}

This is close to the concept of {\em separoid} of Dawid \cite{dawid}.
The standard example of conditional independence is the case of vector spaces where we can interpret $\bx\ \bot_\bz\ \by$ as $\la \bx\ra\cap \la \by\ra\subseteq\la \bz\ra$. In model theory the corresponding concept is ``$\bx$ is {\em independent of $\by$ over $\bz$}" (Definition~\ref{3.14}). We cannot solve the Completeness Question for the class of conditional independence structures, but point out some partial negative results.


\section{Abstract Systems}\label{abs_sys}

We commence with a completely abstract setup. Here we assume no background geometry and the meaning of dependence and independence are purely combinatorial. This is the way these concepts are treated in database theory.

\smallskip


 \index{Atomic!Dependende Logic} {\em Atomic Dependence Logic} is defined as follows.
%
%
	The language of this logic is made  of dependence atoms only. That is, if $\vec{x}$ and $\vec{y}$ are finite sequences of variables, with $\vec{y} \neq \emptyset$ if $\vec{x} \neq \emptyset$, then the formula $\Dep{x}{y}$ is a formula of the language of Atomic Dependence Logic.
	The deductive system $\vdash_\mathcal{D}$ consists of the following set of rules:
\begin{enumerate}[($a{_1.}$)]
	\item $\Dep{x}{x}$;
	\item If $\dep({\vx},{\vy\vz})$,  then $\dep({\vx\vu},{\vy})$;
	\item If $\Dep{x}{y}$ and $\Dep{y}{z}$, then $\Dep{x}{z}$;
	\item If $\Dep{x}{y}$ and $\dep(\vec{x}, \vec{v})$, then $\dep(\vec{x}, \vec{y}  \vec{v})$;
	\item If $\Dep{x}{y}$ and $\vz$ is a permutation of $\vx$, then $\Dep{z}{y}$.

\end{enumerate}
As a degenerate case of ($a_1.$) we admit $\dep(\emptyset, \emptyset)$. For the semantics
let us recall from Example~\ref{teamexample} that  if $I$ is a set of variables, a {\em team with domain $I$} is any set $X$ of functions $s:I\to M$. 
For such $X$ and finite sets $\bx$ and $\by$ of elements of $I$ we define $$\bx\rr\by\mbox{ if and only if } \forall s, s' \in X \; (s(\bx) = s'(\bx) \rightarrow s(\by) = s'(\by)). $$
This defines a dependence structure $$\D_X=(I,\rr),$$ in the sense of Definition~\ref{depdef},   which we call a {\em team dependence structure}.

In plain words, a team is a table of data and $\bx\rr\by$ holds if the data on columns listed in $\bx$ functionally determines the data on columns listed in $\by$.
In the team (table of data) depicted in Figure~\ref{ateam}  the dependence
$\{x_1,x_2,x_3\}\rr\{x_4,x_5\}$ holds but $\{x_2,x_3\}\Rightarrow\{x_5\}$ does not. 
\begin{figure}[h]
\begin{tabular}{|c|c|c|c|c|}
\hline
$x_1$&$x_2$&$x_3$&$x_4$&$x_5$\\
\hline
0&0&1&2&3\\
0&1&1&4&3\\
1&1&1&4&4\\
0&1&0&3&2\\
\hline
\end{tabular}
 \caption{A team.\label{ateam}}
\end{figure}
Functional dependencies have a great importance in database theory (see e.g. \cite[Chapter 8]{zbMATH00839556}). Finding functional dependencies is also an important goal in science in general: what determines the time of descent of a freely falling body? What determines the gender of an offspring? Is tendency to diabetes  hereditary? What causes cancer? What causes global warming? An important approach to all such questions is collecting a large data set $X$ (a team) and trying to locate dependencies in the sense of $\D_X$.

Suppose $V$ is a set of variables.  We can interpret $V$ in a dependence structure $\D=(I,\rr_\D)$ with a mapping  $s:V\to I$ called assignment (into $\D$).  This induces a canonical mapping $\vx\mapsto s(\vx)$ as above. Then for any  finite sequences $\vx$ and $\vy$ of elements of $V$ and any assignment $s$ it is natural to define
$$\D\models_s \dep(\vx,\vy)\mbox{ if and only if }s(\vx) \rr_\D s(\vy).$$ If this is the case, we say that $s$ {\em satisfies} the dependence $\dep(\vx,\vy)$ in $\D$. 

Suppose $\C$ is a class of dependence structures. If every assignment into any dependence structure in $\C$ which satisfies every atom in a set $\Sigma$ of dependence atom, also satisfies the dependence atom $\phi$, we write $\Sigma\models_{\C} \phi$.

%
%
%
%
%
%
%

%
Let $\C_{\mbox{\tiny TD}}$ be the class of all team dependence structures. The following result is known in database theory as Armstrong's Completeness Theorem:	

\begin{theorem}[\cite{DBLP:conf/ifip/Armstrong74}] The Completeness Question for $\C_{\mbox{\tiny TD}}$ has a positive answer i.e. if $\Sigma$ is a set  of dependence atoms and $\phi$ is a a dependence atom, then $$\Sigma\vdash_\mathcal{D}\phi\iff\Sigma\models_{\C_{\mbox{\tiny TD}}} \phi.$$
\end{theorem} 

Inspection of the proof reveals that we can choose $M=\{0,1\}$ and the teams can be limited to consist of just two assignments. This shows how restricted a class can give a positive answer to the Completeness Question for the dependence atoms.
%
	%



{\em Atomic Absolute Independence Logic} is defined as follows.
%
%
The language of this logic is made only of absolute independence atoms defined as follows: suppose $\vec{x}$ is a finite sequence of distinct\footnote{Asking that the variables are distinct is only a technical requirement for the completeness proof below. Obviously, this assumption comes at no conceptual cost.} variables, then $\bota (\vec{x})$ is an {\em absolute independence atom}.
	The deductive system $\vdash_\mathcal{AI}$ consists of the following set of rules:
\begin{enumerate}[($a_2.$)]
	\item $\bota (\emptyset)$;
	\item If $\bota (\vec{x}\vec{y})$, then $\bota (\vec{x})$;
	\item If $\bota (\vec{x})$, then $\bota (\vec{y})$ whenever $\vy$ is a permutation of $\vec{x}$.
	
\end{enumerate}

%
%
	The intuition behind the atom $\bota (\vec{x})$ is that $\vec{x}$ consists of independent elements. That is, each element of $\vec{x}$ is independent of all the other elements of $\vec{x}$. In particular, we ask that each element of $\vec{x}$ does not depend on any other element, i.e. that it is not constant, for in our terminology a constant would be determined (in a trivial way) by any of the other variables.
%

Let $X$ be a team in the sense of Example~\ref{teamexample}. If $x_i \in \bx$, we denote by $\bx -_X x_i$  the set $\left\{ x_j \in \bx \; | \; 
%
%
\exists s\in X(s(x_i)\ne s(x_j))\right\}$. 
Let $\A_X$ be the set of finite subsets $\bx$ of $I$ such that	\[ \forall s, s' \in X \;  \exists s'' \in X \; (s''(x_i) = s(x_i) \wedge s''(\bx -_X x_i) = s'(\bx -_X x_i)) \]
	\[ \;\;\;\;\;\; \text{ and } \;\;\;\;\; \]
	\[ \exists s, s' \in X \; (s(x_i) \neq s'(x_i)). \]
Then $\I_X=(I,\A_X)$ is an absolute independence structure in the sense of Definition~\ref{ainddef}, and we call it a {\em  absolute team independence structure}.

In plain words, $\bx\in\A_X$ holds in a team $X$ (i.e. in an absolute team independence structure $\I_X$) if the data on columns listed in $\bx$ has the property that knowing data on one of the columns gives no hint what the data on the other columns is.
In the team (table of data) depicted in Figure~\ref{aateam}  the absolute independence
$\bota(\{x_1,x_2\})$ holds but $\bota(\{x_1,x_3\})$ does not. 
\begin{figure}[h]
\begin{tabular}{|c|c|c|}
\hline
$x_1$&$x_2$&$x_3$\\
\hline
0&0&1\\
0&1&1\\
1&0&1\\
1&1&0\\
\hline
\end{tabular}
 \caption{Another team.\label{aateam}}
\end{figure}
Absolute independencies have a great importance in statistics. If in the team of Figure~\ref{aateam} $x_1,x_2,x_3$ are random variables, the independence of e.g. $\{x_1,x_2\}$, i.e. the fact that $\{x_1,x_2\}\in\A_X$, means simply that $x_1$ and $x_2$ are independent random variables in the sense of probability theory.

Suppose $V$ is a set of variables.  We can interpret $V$ in an absolute independence structure $\I=(I,\A)$ with a mapping  $s:V\to I$ called assignment (into $\I$).  This induces a canonical mapping $\vx\mapsto s(\vx)$ as above. Then for any  finite sequence $\vx$ of elements of $V$ and any assignment $s$ it is natural to define
$$\I\models_s \bota(\vx)\mbox{ if and only if }s(\vx)\in \A.$$ If this is the case, we say that $s$ {\em satisfies} the absolute independence atom $\bota(\vx)$ in $\I$. 

Suppose $\C$ is a class of absolute independence structures. It should be clear what $\Sigma\models_{\C} \phi$ means for a set $\Sigma$ of absolute independence atoms and a single absolute independence atom $\phi$.

%
Let $\C_{\mbox{\tiny ATM}}$ be the class of all absolute team independence structures.
	\begin{theorem}  The Completeness Question for $\C_{\mbox{\tiny ATM}}$ has a positive answer i.e. if $\Sigma$ is a set  of absolute independence atoms and $\phi$ is a an absolute  independence atom, then $$\Sigma\vdash_{\mathcal{AI}}\phi\iff\Sigma\models_{\C_{\mbox{\tiny ATM}}} \phi.$$
\end{theorem}

	\begin{proof} Soundness is obvious (for rule ($a_2.$) notice that the quantifier for all $x \in \vec{x}$ is vacuous).
Regarding completeness, suppose $\Sigma \nvdash \bota (\vec{x})$. Notice that, because of rule ($a_2.$), $\vec{x} \neq \emptyset$. Let $\vec{x} = (x_{0}, ..., x_{n-1})$.
	Let ${M}=\left\{ 0, 1 \right\}$. Define $X = \left\{ s_t \; | \; t \in 2^{\omega} \right\}$ to be the set of assignments which give all the possible combinations of $0$s and $1$s to all the variables but $x_{0}$ and which at $x_{0}$ are such that 
		\[ \begin{array}{rcl} 
			\;\;\;\;\;\;\;\;\;\;\;\;\;\; & s_t(x_{0}) = 0	& \;\;\;\;\;\; \text{ if }	\vec{x} = ( x_{0} )	
		\end{array}\]
		\[ \begin{array}{rcl} 
			\;\;\;\;\;\;\;\;\;\;\;\;\;\; & s_t(x_{0}) = p(s_t(x_{1}), ..., s_t(x_{n-1}))	& \;\;\;\;\;\; \text{ if }	\vec{x} \neq ( x_{0} )
		\end{array}\]
for all $t \in 2^{\omega}$, where $p: M^{<\omega} \rightarrow M$ is the function which assigns $1$ to the  sequences with an odd numbers of $1$s and $0$ to the sequences with an even numbers of $1$s.

Let $\I=\I_X$.
	We claim that $\I \not\models \bota (\vec{x})$. If $\vec{x} = ( x_{0})$, then  we have that for all $t \in 2^{\omega}$, $s_t(x_{0}) = 0$. Thus, there is $x \in \vec{x}$ for which there are no $s, s' \in X$ such that $s(x) \neq s'(x)$. On the other hand, if $\vec{x} \neq ( x_{0})$, then for all $t \in 2^{\omega}$, $s_t(x_{0}) = p(s_t(x_{1}), ..., s_t(x_{n-1}))$. Notice that in this case $n \geq 2$. Let $t, d \in 2^{\omega}$ be such that $s_t(x_{1}) = 0$, $s_d(x_{1}) = 1$ and $s_t(x_{i}) = s_d(x_{i})$ for every $i \in \left\{ 2, ..., n-1 \right\}$. Clearly 
	\[ p(s_t(x_{1}), ..., s_t(x_{n-1})) \neq p(s_d(x_{1}), ..., s_d(x_{n-1})).\]
	Suppose that $\I \models \bota (\vec{x})$. Then there exists $f \in 2^{\omega}$ such that 
	\[ s_f(x_{0}) = s_t(x_{0}) \text{ and } s_f(\vec{x} -_X x_{0}) = s_d(\vec{x} -_X x_{0})). \]
Notice that under $X$ we have that $\vec{x} -_X x_{0} = (x_{1}, ..., x_{n-1})$, thus
	\[ s_f(\vec{x} -_X x_{0}) = (s_f(x_{1}), ..., s_f({x_{n-1}})) \]
and	
	\[ s_d(\vec{x} -_X x_{0}) = (s_d(x_{1}), ..., s_d({x_{n-1}})). \]
	Hence 
	\[	\begin{array}{rcl}
				p(s_d(x_{1}), ..., s_d(x_{n-1})) & = &  p(s_f(x_{1}), ..., s_f(x_{{n-1}})) \\
													   & = &  s_f(x_{0}) \\
													   & = &  s_t(x_{0}) \\
													   & = &  p(s_t(x_{1}), ..., s_t(x_{n-1})),
\end{array} \]
which is a contradiction.
\smallskip

\noindent	Let now $\bota (\vec{v}) \in \Sigma$, we want to show that $\I \models \bota (\vec{v})$. Notice that if $\vec{v} = \emptyset$, then $\I\models_X \bota (\vec{v})$. Thus let $\vec{v} = (v_{0}, ..., v_{c-1})  \neq \emptyset$. 
	We make a case distinction on $\vec{v}$. 
\smallskip

\noindent	{\bf Case 1.} $x_{0} \notin \vec{v}$. 
	Let $v \in \vec{v}$. Because of the assumption, $v \neq x_{0}$ and $x_{0} \notin \vec{v} -_X v$. Thus for every $t, d \in 2^{\omega}$ clearly there is $f \in 2^{\omega}$ such that
	\[ s_f(v) = s_t(v) \text{ and } s_f(\vec{v} -_X v) = s_d(\vec{x} -_X v). \]
\smallskip	

\noindent {\bf Case 2.} $x_{0} \in \vec{v}$. 
\smallskip	

\noindent{\bf Subcase 2.1.} $ \vec{x} - \vec{v} \,\neq \emptyset$. 
	Notice that $\vec{x} \neq ( x_{0} )$ because if not then $\vec{x} - \vec{v} = ( x_{0} )$ and so $x_{0} \notin \vec{v}$. Hence for every $t \in 2^{\omega}$ we have that
		\[s_t(x_{0}) = p(s_t(x_{1}), ..., s_t(x_{n-1})) .\]
	Suppose, without loss of generality, that $\vec{v} = (x_{0}, v_{1}, ..., v_{c-1})$ and let $\vec{x}' =  \vec{x} \cap \vec{v} = (u_{0}, ..., u_{m-1})$ and $z \in \vec{x} - \vec{v}$. Let $v \in \vec{v}$.
\smallskip	
	
\noindent	{\bf Subcase 2.1.1.} $v \neq x_{0}$.
	Let $k \in \left\{ 1, ..., c-1 \right\}$ and $v = v_{k}$. Let $t, d \in 2^{\omega}$ and let $f \in 2^{\omega}$ be such that:
	\begin{enumerate}[i)]
	 	\item $s_f(v_{k}) = s_t(v_{k})$;
		\item $s_f(v_{i}) = s_d(v_{i})$ for every $i \in \left\{ 1, ..., k-1, k+1, ..., c-1 \right\}$;
		\item $s_f(u) = 0$ for every $u \in \vec{x} - \vec{x}' z$;
		\item $s_f(z) = 0$, if $p(s_f(u_{0}), ..., s_f(u_{m-1})) = s_d(x_{0})$ and $s_f(z) = 1$ otherwise.		
\end{enumerate}
Then $f$ is such that
	\[ s_f(v_{k}) = s_t(v_{k}) \]
and
	\[ (s_f(x_{0}), s_f(v_{1}), ..., s_f(v_{k-1}), s_f(v_{k+1}), ..., s_f(v_{c-1}))  = \]
	\[ (s_d(x_{0}), s_d(v_{1}), ..., s_d(v_{k-1}), s_d(v_{k+1}), ..., s_d(v_{c-1})). \]
\smallskip	

\noindent	{\bf Subcase 2.1.2.} $v = x_{0}$.
	Let $t, d \in 2^{\omega}$ and let $f \in 2^{\omega}$ be such that:

	\begin{enumerate}[i)]
		\item $s_f(v_{i}) = s_d(v_{i})$ for every $i \in \left\{ 1, ..., c-1 \right\}$;
		\item $s_f(u) = 0$ for every $u \in \vec{x} - \vec{x}'  z$;
		\item $s_f(z) = 0$, if $p(s_f(u_{0}), ..., s_f(u_{m-1})) = s_t(x_{0})$ and $s_f(z) = 1$ otherwise.		
\end{enumerate}
Then $f$ is such that
	\[ s_f(x_{0}) = s_t(x_{0}) \]
and
	\[ (s_f(v_{1}), ..., s_f(v_{c-1})) = (s_d(v_{1}), ..., s_d(v_{c-1})). \]
\smallskip	
	
\noindent	{\bf Subcase 2.2.} $\vec{x} \subseteq \vec{v}$. 
This case is not possible. Suppose indeed it is, then by rule ($c_2.$) we can assume that $\vec{v} = \vec{x} \vec{v}'$ with $\vec{v}' \subseteq \mathrm{Var} - \vec{x}$. Thus by rule ($b_2.$) we have that $\Sigma \vdash \bota (\vec{x})$, contrary to our assumption.
\smallskip	
	
\noindent This concludes the proof of the theorem.
	
\end{proof}

We can observe that for the absolute team independence structures needed in the proof we can choose $M=\{0,1\}$ and the team can be chosen to be of cardinality $2^n$, where $n$ is the number of variables in $\Sigma\cup\{\phi\}$.
\bigskip



 \index{Atomic!Independende Logic} {\em Atomic Independence Logic} is defined as follows.
%
%
	The language of this logic is made  of independence atoms as defined in the Introduction. That is, if $\vec{x}$ and $\vec{y}$ are finite sequences from a set $V$ of variables, then the atomic formula $\vec{x} \boto \vec{y}$ is a formula of the language of Atomic Independence Logic.
%
%
The deductive system $\vdash_\mathcal{I}$ consists of the following set of rules:
\begin{enumerate}[($a_3.$)]
	\item $\vec{x} \boto \emptyset$;
	\item If $\vec{x} \boto \vec{y}$, then $\vec{y} \boto \vec{x}$;
	\item If $\vec{x} \boto \vec{y}\vec{z}$, then $\vec{x} \boto \vec{y}$;
	\item If $\vec{x} \boto \vec{y}$ and $\vec{x}\vec{y} \boto \vec{z}$, then $\vec{x} \boto \vec{y}\vec{z}$;
	\item If $x \boto x$, then $x \boto \vec{y}$;
	\item If $\vec{x} \boto \vec{y}$, then $\vec{u} \boto \vec{v}$ whenever $\vu$ and $\vv$ are permutations of $\vec{x}$ and $\vec{y}$ respectively;
			\item If $\vec{x} y \vec{z} \boto \vec{w}$, then $\vec{x} y y \vec{z} \boto \vec{w}$.
	
\end{enumerate}

 We can interpret $V$ in an independence structure $\S=(I,\bot)$ with an assignment mapping  $s:V\to I$.  This induces a canonically $\vx\mapsto s(\vx)$. Then for any  finite sequence $\vx$ of elements of $V$ and any assignment $s$ we define
$$\S\models_s \vx\boto\vy\mbox{ if and only if }s(\vx)\boto s(\vy).$$ We then say that $s$ {\em satisfies} the  independence atom $\vx\boto\vy$ in $\S$.

Let $\C_{\mbox{\tiny TI}}$ be the class of {all}  team independence structures as defined in Example~\ref{teamexample}.
The following Completeness Theorem is a known result in  statistics:	
	
	\begin{theorem}[\cite{geiger1991axioms} and \cite{galliani_and_vaananen}]  The Completeness Question for $\C_{\mbox{\tiny TI}}$ has a positive answer i.e. if $\Sigma$ is a set  of  independence atoms and $\phi$ is an   independence atom, then $$\Sigma\vdash_{\mathcal{I}}\phi\iff\Sigma\models_{\C_{\mbox{\tiny TI}}} \phi.$$
\end{theorem}

Inspection of the proof reveals that  for the team needed in the above proof we can choose $M=\{0,1\}$ and the team can be chosen to be of cardinality $2^n$, where $n$ is the number of variables in $\Sigma\cup\{\phi\}$.		
\bigskip

	\index{Atomic!Conditional Independence Logic} {\em Atomic Conditional Independence Logic} is defined as follows.
%
%
	The language of this logic is made  of conditional independence atoms only. That is, if $\vec{x}$, $\vec{y}$  and $\vec{z}$ are finite sequences of variables, then the formula $\vec{x} \botc{\vec{z}} \vec{y}$ is a formula of the language of Atomic Conditional Independence Logic.
The deductive system $\vdash_\mathcal{CI}$ consists of the following set of rules:
\begin{enumerate}[($a_4.$)]
	\item $\vec{x} \botc{\vec{x}} \vec{y}$;
	\item If $\vec{x} \botc{\vec{z}} \vec{y}$, then $\vec{y} \botc{\vec{z}} \vec{x}$;
	\item If $\vec{x}  \vec{x}'  \botc{\vec{z}} \vec{y}  \vec{y}'$, then $\vec{x} \botc{\vec{z}} \vec{y}$;
	\item If $\vec{x} \botc{\vec{z}} \vec{y}$, then $\vec{x}  \vec{z}  \botc{\vec{z}} \vec{y}  \vec{z}$;
	\item If $\vec{x} \botc{\vec{z}} \vec{y}$ and $\vec{u} \botc{\vec{z}, \vec{x}} \vec{y}$, then $\vec{u} \botc{\vec{z}} \vec{y}$;
	\item If $\vec{y} \botc{\vec{z}} \vec{y}$ and $\vec{z}  \vec{x} \botc{\vec{y}} \vec{u}$, then $\vec{x} \botc{\vec{z}} \vec{u}$;
	\item If $\vec{x} \botc{\vec{z}} \vec{y}$ and $\vec{x}  \vec{y} \botc{\vec{z}} \vec{u}$, then $\vec{x} \botc{\vec{z}} \vec{y}  \vec{u}$;
	\item If $\vec{x} \botc{\vec{z}} \vec{y}$, then $\vec{x} \botc{\vec{z}} \vec{y}$ whenever $\vu$, $\vv$ and $\vw$ are permutations of $\vec{x}$, $\vec{z}$, and $\vec{y}$ respectively.

\end{enumerate}

 We can interpret $V$ in a conditional independence structure $\S=(I,\bot)$ with an assignment mapping  $s:V\to I$.  This induces a canonically $\vx\mapsto s(\vx)$ and we can define
$$\S\models_s \vx\botc{\vz}\vy\mbox{ if and only if }s(\vx)\botc{s(\vz)} s(\vy).$$ We then say that $s$ {\em satisfies} the conditional independence atom $\vx\botc{\vz}\vy$ in $\S$.

For the semantics
we use again the concept of team  from Example~\ref{teamexample}. Suppose $X$ is a team. We define that  $\bx \botc{\bz} \by$ holds for finite sets $\bx$, $\by$ and $\bz$ of variables if 
		\[ \forall s, s' \in X (s(\bz) = s'(\bz) \rightarrow \exists s'' \in X (s''(\bz) =s(\bz) \wedge  s''(\bx) = s(\bx)  \wedge  s''(\by) = s'(\by))). \]

This defines a conditional independence structure $\mathcal{CI}_X=(I,\bot)$   which we call a {\em  conditional team independence structure}.

%

Let $\C_{\mbox{\tiny TCI}}$ be the class of all conditional team independence structures $\mathcal{CI}_X$.
 The system Atomic Conditional Independence Logic is sound in the sense that if $\Sigma$ is a set  of absolute independence atoms and $\phi$ is a an absolute  independence atom, then $$\Sigma\vdash_{\mathcal{CI}}\phi\Rightarrow\Sigma\models_{\C_{\mbox{\tiny TCI}}} \phi.$$ The converse is not true.
	Parker and Parsaye-Ghomi \cite{Parker:1980:IIE:582250.582259} proved that it is not possible to find a \emph{finite} complete axiomatization for the conditional independence atoms. Furthermore, in \cite{herrmann_undec} and \cite{herrmann_undec_corrig} Hermann proved that the consequence relation between these atoms is undecidable. It is, a priori, obvious  that there is {\em some} recursive axiomatization for the conditional independence atoms true in all team independence structures, because we can reduce the whole question to first-order logic with extra predicates and then appeal to the Completeness Theorem of first-order logic. In \cite{naumov} Naumov and Nicholls developed an explicit recursive  axiomatization of conditional independence atoms, but it cannot be put in the form of axioms such as $(a_4.)-(h_4.)$.


\section{Dependence and Independence in Pregeometries}\label{preg}

In this section we consider independence and dependence structures arising from  geometric structures. Teams and databases do not in general have such structure, but algebraic structures may have as well as, in some cases, models of first-order theories.

\subsection{Closure Operator Atomic Dependence Logic}

	\begin{definition}\label{def_clo_op} Let $M$ be a set and $\mathrm{cl}: \mathcal{P}(M) \rightarrow \mathcal{P}(M)$ an operator on the power set of $M$. We say that $\mathrm{cl}$ is a \emph{closure operator} and that $(M, \mathrm{cl})$ is a \emph{closure system} if for every $A, B \subseteq M$ the following conditions are satisfied:
		\begin{enumerate}[i)]
			\item $A \subseteq \mathrm{cl}(A)$;
			\item If $A \subseteq B$ then $\mathrm{cl}(A) \subseteq \mathrm{cl}(B)$;
			\item $\mathrm{cl}(A) = \mathrm{cl}(\mathrm{cl}(A))$.
\end{enumerate}
		
\end{definition}

\noindent Given a closure system $(M, \mathrm{cl})$, we say that $A \subseteq M$ is closed if $\mathrm{cl}(A) = A$.

	\begin{example} Let $\mathcal{A}$ be an algebra, and for every $B \subseteq A$ let $[B]$ be the subalgebra of $\mathcal{A}$ generated by $B$. Then $[ \; ]: \mathcal{P}(A) \rightarrow \mathcal{P}(A)$ is a closure operator. 
				
\end{example}

	\begin{example} Let $(X, \tau)$ be a topological space, and for every $Y \subseteq X$ let $\overline{Y}$ be smallest closed subset of $X$ that contains $Y$. Then $\overline{\phantom{Y}}: \mathcal{P}(A) \rightarrow \mathcal{P}(A)$ is a closure operator. 
		
\end{example}

	\begin{example} Let $\mathcal{M}$ be a first-order structure in the signature $L$ and $A \subseteq M$. We say that $b$ is \emph{algebraic} over $A$ if there is a first-order $L$-formula $\phi (v, \vec{w})$ and $\vec{a} \in A$ such that $\mathcal{M} \models \phi (b, \vec{a})$ and $\phi (\mathcal{M}, \vec{a}) = \left\{ m \in M \; | \; \mathcal{M} \models \phi (m, \vec{a}) \right\}$ is finite. Let $\mathrm{acl}(A) = \left\{ m \in M \; | \; m \text{ is algebraic over } A \right\}$. Then $\mathrm{acl}: \mathcal{P}(M) \rightarrow \mathcal{P}(M)$ is a closure operator.
		
\end{example}

	{\em Closure Operator Atomic Dependence Logic} is defined as follows. 
The syntax and deductive system of this logic are the same as those of Atomic Dependence Logic (see Section \ref{abs_sys}).
%
%
%
%
Given a closure system $(M, \mathrm{cl})$, we define a dependence structure $$\D_{(M, \mathrm{cl})}=(M,\rr)$$ as follows: If $\bx$ and $\by$ are finite subsets of $M$, then 
\[\bx\rr\by\mbox{ if and only if } \by \subseteq \mathrm{cl}(\bx). \] We call $\D_{(M, \mathrm{cl})}$ a {\em closure operator dependence structure}.

%
%
%
 
%
%

	\begin{theorem}\label{compl_clo_op_dep_logic} 
Suppose $\C$ is a class of   dependence structures  such that there exists 
$\D_{(M, \mathrm{cl})} \in \C$ such that in $(M, \mathrm{cl})$ we have $\emptyset\ne\mathrm{cl}(\emptyset) \ne M$.
	 Then the Completeness Question for $\C$ has a positive answer i.e. if $\Sigma$ is a set  of dependence atoms and $\phi$ is a dependence atom, then $$\Sigma\vdash_{\mathcal{D}}\phi\iff\Sigma\models_{\C} \phi.$$
\end{theorem}

		\begin{proof} Soundness is easy. Regarding completeness, suppose $\Sigma \nvdash \dep(\vec{x}, \vec{y})$. Let $V = \left\{ z \in \mathrm{Var} \;|\; \Sigma \vdash \dep(\vec{x}, z) \right\}$ and $W = \mathrm{Var} - V$. Remember that in the definition of the syntax of this system we ask that for every atom $\dep(\vec{v}, \vec{w})$ we have that $\vec{w} \neq \emptyset$ whenever $\vec{v} \neq \emptyset$. Thus, $\dep(\vec{x}, \vec{y})$ is so that $\vec{y} \neq \emptyset$, because otherwise $\vec{x}, \vec{y} = \emptyset$, and so, by the admitted degenerate case of rule ($a_1.$), we have that $\Sigma \vdash \dep(\vec{x}, \vec{y})$. Furthermore $\vec{y} \cap W \neq \emptyset$, which can be seen as follows. If $\vec{y} \cap W = \emptyset$, then for every $y \in \vec{y}$ we have that $\Sigma \vdash \dep(\vec{x}, y)$, and so by rule ($d_1.$) we have that $\Sigma \vdash \dep(\vec{x}, \vec{y})$. 
\smallskip

\noindent	By assumption there is $\D_{(M, \mathrm{cl})} \in \C$ such that $\emptyset\ne\mathrm{cl}(\emptyset) \ne M$ in $(M, \mathrm{cl})$, so there are $a, b \in M$ with $a \in \mathrm{cl}(\emptyset)$ and $b \notin \mathrm{cl}(\emptyset) = \mathrm{cl}(\left\{ a \right\})$. Let $s$ be the following assignment:
			
			\[s(v) = \begin{cases} a \;\;\;\;\; \text{ if } v \in V \\
								   b \;\;\;\;\; \text{ if } v \in W.

		\end{cases}\]
	We claim that $\D_{(M, \mathrm{cl})} \not\models_s \dep(\vec{x}, \vec{y})$. In accordance to the semantic we have to show that there is $y \in \vec{y}$ such that $s(y) \notin \mathrm{cl}(\left\{ s(x) \; | \; x \in \vec{x} \right\})$. Let $y \in \vec{y} \cap W$, then 
			\[s(y) = b \notin \mathrm{cl}(\left\{ a \right\}) = \mathrm{cl}(\left\{ s(x) \; | \; x \in \vec{x} \right\}), \]
	because for $x \in \vec{x}$ we have that $\Sigma \vdash \dep(\vec{x}, x)$. Indeed by rule ($a_1.$) $\vdash \dep(\vec{x}, \vec{x})$ and so by rule ($b_1.$) $\vdash \dep(\vec{x}, x)$. Notice that in the case $\vec{x} = \emptyset$, we have that
		   \[ s(y) = b \notin \mathrm{cl}(\left\{ a \right\}) = \mathrm{cl}(\emptyset) = \mathrm{cl}(\left\{ s(x) \; | \; x \in \vec{x} \right\}). \]
			Let now $\dep(\vec{x}', \vec{y}') \in \Sigma$, we want to show that $\D_{(M, \mathrm{cl})}\models_s \dep(\vec{x}', \vec{y}')$. If $\vec{y}' = \emptyset$ then also $\vec{x}' = \emptyset$, and so trivially $\D_{(M, \mathrm{cl})} \models_s \dep(\vec{x}', \vec{y}')$. Having noticed this, for the rest of the proof we assume $\vec{y}' \neq \emptyset$.
\smallskip
		
\noindent	{\bf Case 1.} $\vec{x}' = \emptyset$.
			Suppose that $\D_{(M, \mathrm{cl})}\not\models_s \dep(\emptyset,  \vec{y}')$, then there exists $y' \in \vec{y}'$ such that $s(y') = b$, so $\Sigma \nvdash \dep(\vec{x}, y')$. Notice though that $\Sigma \vdash \dep(\emptyset, \vec{y}')$, so by rule ($b_1.$) $\Sigma \vdash \dep(\emptyset, y')$ and hence again by rule ($b_1.$) $\Sigma \vdash \dep(\vec{x}, y')$.
\smallskip

\noindent	{\bf Case 2.} $\vec{x}' \neq  \emptyset$ and $\vec{x}' \subseteq V$. 
			If this is the case, then 
		\[	\begin{array}{rcl}
					\forall x' \in \vec{x}' \;\; \Sigma \vdash \dep(\vec{x}, x') 
							& \Longrightarrow 		& \Sigma \vdash \dep(\vec{x}, \vec{x}') \;\;\;\hspace{-0.05em} \text{[by rule ($d_1.$)]} \\
							& \Longrightarrow 		& \Sigma \vdash \dep(\vec{x}, \vec{y}') \;\;\; \text{[by rule ($c_1.$)]}\\
							& \Longrightarrow 		&  \forall y' \in \vec{y}' \;\; \Sigma \vdash \dep(\vec{x}, y') \;\;\; \text{[by rule ($b_1.$)]}\\
							& \Longrightarrow 		& \vec{y}' \subseteq {V}.
				\end{array} \]
	If $\vec{x}' \subseteq V$ then for every $x' \in \vec{x}'$ we have that $s(x') = a$, so $\mathrm{cl}(\left\{ s(x') \; | \; x' \in \vec{x}' \right\}) = \mathrm{cl}(\left\{ a \right\})$. Let $y' \in \vec{y}'$, then we have that $s(y') = a$ and clearly $a \in \mathrm{cl}(\left\{ a \right\})$.
\smallskip

\noindent	{\bf Case 3.} $\vec{x}' \cap W \neq \emptyset$. 
			If this is the case, then there exists $w \in \vec{x}'$ such that $\Sigma \nvdash \dep(\vec{x}, w)$. Thus, we have $w \in \vec{x}'$ such that $s(w) = b$, and so $\mathrm{cl}(\left\{ s(x') \; | \; x' \in \vec{x}' \right\}) = \mathrm{cl}(\left\{ b \right\})$. Let now $y' \in \vec{y}'$, then either $s(y') = a$ or $s(y') = b$, but in both cases we have that $s(y') \in \mathrm{cl}(\left\{ b \right\}) = \mathrm{cl}(\left\{ s(x') \; | \; x' \in \vec{x}' \right\})$.
\smallskip	
			
\noindent This concludes the proof of the theorem.

\end{proof}

\subsection{Pregeometries}\label{pregeometries}

	Noticing various similarities in which the notion of dependence occurs in linear algebra, field theory and graph theory, in the mid 1930's, Hassler Whitney \cite{whitney} and Bartel Leendert van der Waerden \cite{waerden} independently identified  a few conditions capable to subsume all these cases of dependence. This led to the definition of the notion of {\em abstract dependence relation}, also known as {\em matroid}. In the 1970's, Giancarlo Rota  and Henry H. Crapo \cite{rota} introduced the term {\em pregeometry}. Although strictu sensu the two terms are synonymous, sometimes mathematicians refer to finite pregeometries as matroids. Finite matroids can be characterized in several equivalent ways, but these equivalences fail in the infinite setting. Thus, the general definition of a pregeometry generalizes only one of the aspects of these finite objects. In the model-theoretic community the term pregeometry is preferred, probably because of the focus on infinite structures. 
	
	\begin{definition}\label{def_pregeo} Let $M$ be a set and $\mathrm{cl}: \mathcal{P}(M) \rightarrow \mathcal{P}(M)$ a closure operator on the power set of $M$. We say that $(M, \mathrm{cl})$ is a \emph{pregeometry} if for every $A, B \subseteq M$ and $a, b \in M$ the following conditions are satisfied:
		\begin{enumerate}[i)]
			\item if $a \in \mathrm{cl}(A \cup \left\{ b \right\}) - \mathrm{cl}(A)$, then $b \in \mathrm{cl}(A \cup \left\{ a \right\})$ \; [Exchange Principle];
			\item if $a \in \mathrm{cl}(A)$, then $a \in \mathrm{cl}(A_0)$ for some $A_0 \subseteq_{\omega} A$ \; [Finite Character].
\end{enumerate}
		
\end{definition} 
	
	\begin{example}\label{span_is_pregeo} Let $\mathbb{K}$ be a field and $\mathbb{V}$ be a vector space over $\mathbb{K}$. For every $A \subseteq V$ let $\langle A \rangle$ be the smallest subspace of $\mathbb{V}$ containing $A$, i.e. the subspace of $\mathbb{V}$ spanned by $A$. Then $(V, \langle \, \rangle)$ is a pregeometry. 
		
\end{example}

	\begin{example}\label{acl_is_pregeo} Let $\mathbb{K}$ be an algebraically closed field and, for $a \in K$ and $A \subseteq K$, let $a \in \mathrm{acl}(A)$ if $a$ is algebraic over the subfield of $\mathbb{K}$ generated by $A$. Then $(K, \mathrm{acl})$ is a pregeometry, see for example \cite{milne}.
		
\end{example}

	\begin{definition} Let $(M, \mathrm{cl})$ be a pregeometry. 
		\begin{enumerate}[i)]
			\item We say that $(M, \mathrm{cl})$ is a {\em geometry} if $\mathrm{cl}(\emptyset) = \emptyset$ and $\mathrm{cl}(\left\{ m \right\}) = \left\{ m \right\}$ for all $m \in M$.
			\item We say that $(M, \mathrm{cl})$ is {\em trivial} if $\mathrm{cl}(A) = \bigcup_{a \in A}\mathrm{cl}(\left\{ a \right\})$ for any $A \subseteq M$.
		\end{enumerate}
		
\end{definition}

	\begin{definition} Let $(M, \mathrm{cl})$ be a pregeometry and $A \subseteq M$. We say that $A$ is independent if for all $a \in A$ we have $a \notin \mathrm{cl}(A - \left\{ a \right\})$. Let $(M, \mathrm{cl})$ be a pregeometry and $B \subseteq A \subseteq M$. We say that $B$ is a \emph{basis} for $A$ if $B$ is independent and $A \subseteq \mathrm{cl}(B)$. 
		
\end{definition} \index{basis}

 The following lemma is well-known (see e.g. \cite[Lemma~8.1.3]{marker}).

	\begin{lemma} Let $(M, \mathrm{cl})$ be a pregeometry and $A, B, C \subseteq M$ with $A \subseteq C$ and $B \subseteq C$. If $A$ and $B$ are bases for $C$, then $|A| = |B|$.
				
\end{lemma}

	\begin{definition} Let $(M, \mathrm{cl})$ be a pregeometry and $A \subseteq M$. The \emph{dimension} of $A$ is the cardinality of a basis for $A$. We let $\mathrm{dim}(A)$ denote the dimension of $A$. 
		
\end{definition} \index{dimension}

	If $(M, \mathrm{cl})$ is a pregeometry and $C \subseteq M$, we also consider the \emph{localization} $\mathrm{cl}_C(A) = \mathrm{cl}(C \cup A)$ for $A \subseteq M$. It is easy to see that $(M, \mathrm{cl}_C)$ is also a pregeometry.

	\begin{definition} Let $(M, \mathrm{cl})$ be a pregeometry and $A, C \subseteq M$. We say that $A$ is independent over $C$ if $A$ is independent in $(M, \mathrm{cl}_C)$ and that $B \subseteq A$ is a basis for $A$ over $C$ if $B$ is a basis for $A$ in $(M, \mathrm{cl}_C)$. We let $\mathrm{dim}(A/C)$ be the dimension of $A$ in $(M, \mathrm{cl}_C)$ and call $\mathrm{dim}(A/C)$ the dimension of $A$ over $C$.
		
\end{definition}

	The notion of dimension that we have been dealing with allows us to define an independence relation with many desirable properties.

	\begin{definition}\label{3.14} Let $(M, \mathrm{cl})$ be a pregeometry, $A, B, C \subseteq M$. We say that $A$ is independent of $C$ over $B$ if for every $\vec{a} \in A$ we have $\mathrm{dim}(\vec{a}/B \cup C) = \mathrm{dim}(\vec{a}/B)$. In this case we write $A \clindep[B] C$.
		
\end{definition}

	\begin{lemma}\label{prop_clindep} Let $(M, \mathrm{cl})$ be a pregeometry and $A, B, C, D \subseteq M$. Then
	\begin{enumerate}[i)]
	\item $A \clindep[A] B$ \; [Existence];
	\item if $A \clindep[C] B$ and $D \subseteq B$, then $A \clindep[C] D$ \; [Monotonicity];
	\item $A \clindep[C] B \cup D$ if and only if $A \clindep[C] B$ and $A \clindep[C \cup B] D$ \; [Transitivity];
	\item $A \clindep[C] B$ if and only if $A \clindep[C] B_0$ for all finite $B_0 \subseteq B$ \; [Finite Character];
	\item if $A \clindep[C] B$, then $B \clindep[C] A$ \; [Symmetry];
	\item if $A \clindep[D] B$ and $A \cup B \clindep[D] C$, then $A \clindep[D] B \cup C$ \; [Exchange];
	\item if $A \clindep[B] A$, then $A \clindep[B] E$ for any $E \subseteq M$ \; [Anti-Reflexivity].
	
\end{enumerate}

\end{lemma}  

\begin{proof} See \cite{grossberg}.

\end{proof}

	The following well-known lemma will be relevant in the proof of Theorem \ref{compl_pregeom_atom_indep_logic}.

\begin{lemma}\label{lemma_indep_seq} Let $(M, \mathrm{cl})$ be a pregeometry $C \subseteq M$ and $\left\{ a_i \; | \; i \in I \right\} \subseteq M$ an independent set over $C$. Then for all $A, B \subseteq_{\omega} \left\{ a_i \; | \; i \in I \right\}$ with $A \cap B = \emptyset$ we have that $A \clindep[C] B$.
		
\end{lemma} 

 	Our focus in this paper will be on classes of pregeometries in which there exists $(M, \mathrm{cl})$ with the following  three properties:
\begin{enumerate}[(P1)]
	\item $\mathrm{cl}(\emptyset) \neq \emptyset$;
	\item for every independent $D_0 \subseteq_{\omega} M$, $\mathrm{cl}(D_0) \neq \bigcup_{D \subsetneq D_0} \mathrm{cl}(D)$;	
	\item $\mathrm{dim}(M) \geq \omega$.
	
\end{enumerate}
Notice that conditions (P1) and (P2) put some relevant (but reasonable) restrictions on the pregeometry of $(M, \mathrm{cl})$. Indeed, condition (P1) prohibits that the pregeometry is a geometry, while condition (P2) can be seen as a strong form of non-triviality. It is easy to see that in the case of vector spaces and algebraically closed conditions (P1) and (P2) are alway satisfied.

\subsection{Pregeometry Atomic Independence Logic}

\index{Pregeometry Atomic!Absolute Independence Logic} {\em Pregeometry Atomic Absolute Independence Logic} is defined as follows. 
The syntax and deductive system of this logic are the same as those of Atomic Absolute Independence Logic (see Section \ref{abs_sys}). 
Given a pregeometry $(M, \mathrm{cl})$ we   define an absolute independence structure $$\I_{(M, \mathrm{cl})}=(M,\A)$$ by letting 
$\A$ consists of all finite subsets $\bx$ of $M$ such that $$\forall y \in \bx 
( y \notin \mathrm{cl}(\bx - \left\{ y \right\})). $$  We call $\I_{(M, \mathrm{cl})}$ an {\em absolute pregeometry independence structure}. 
Semantics is defined by  $$\I_{(M, \mathrm{cl})} \models_s \bota (\vec{x})\mbox{ iff } s(\vx)\in\A.$$

%
%


%
%
%
%
%

	\begin{theorem}\label{compl_PGAAIndL} 
Suppose $\C$ is a class of  absolute  independence structures including a structure $\I_{(M, \mathrm{cl})}$ with $(M, \mathrm{cl})$ satisfying (P1)-(P3).
	 Then the Completeness Question for $\C$ has a positive answer i.e. if $\Sigma$ is a set  of absolute  independence atoms and $\phi$ is an absolute independence atom, then $$\Sigma\vdash_{\mathcal{AI}}\phi\iff\Sigma\models_{\C} \phi.$$
\end{theorem}

	\begin{proof} Soundness is obvious.
Regarding completeness, suppose $\Sigma \nvdash \bota (\vec{x})$. Notice that, because of rule ($a_2.$), $\vec{x} \neq \emptyset$. Let $\vec{x} = (x_{0}, ..., x_{n-1})$. 
By assumption there exists $\I_{(M, \mathrm{cl})} \in \mathbf{C}$ such that $(M, \mathrm{cl})$ has the properties (P1), (P2) and (P3). Let then $e \in \mathrm{cl}(\emptyset)$, $B = \left\{ a_i \; | \; i < \omega \right\} \subseteq M$ an independent set and $D_0 = \left\{ a_i \; | \; i < n \right\}$. Let $\left\{ w_i \; | \; i < n \right\}$ be an injective enumeration of $\mathrm{Var} - \left\{ x_0 \right\}$ such that $w_i = x_{i+1}$ for $i < n$. Let $s$ be the following assignment:
	\begin{enumerate}[i)]
		\item $s(w_i) = a_i$ for every $v \in \mathrm{Var} - \left\{ x_0 \right\}$,
		\item $s(x_{0}) = d$,
\end{enumerate}
where $d$ witnesses property (P2) with respect to the set $D_0$, i.e. $d \in \mathrm{cl}(D_0)$ and $d \not\in \bigcup_{D \subsetneq D_0} \mathrm{cl}(D)$ (notice that if $n = 1$, then d $\in \mathrm{cl}(\emptyset)$). Obviously, $\I_{(M, \mathrm{cl})}\not\models_s \bota (\vec{x})$. Let now $\bota (\vec{v}) \in \Sigma$, we want to show that $\I_{(M, \mathrm{cl})} \models_s \bota (\vec{v})$.
	\smallskip

\noindent {\bf Case 1.} $\vec{x} - \vec{v} \,\neq \emptyset$. Let $v \in \vec{v}$, there are three possibilities.
\smallskip

\noindent	{\em Case A.} $v \not\in \vec{v} \cap \vec{x}$.
	Let $v = w_j$, then 
	\[ s(v) \not\in \mathrm{cl}(\left\{ s(w_i) \; | \; j \neq i < \omega \right\}) = \mathrm{cl}(\left\{ s(w_i) \; | \; j \neq i < \omega \right\} \cup \left\{ d \right\}). \] 
Thus, $s(v) \notin \mathrm{cl}(\left\{ s(z) \; | \; z \in \vec{v} \right\} - \left\{ s(v) \right\})$.
\smallskip

\noindent	{\em Case B.} $v = x_0$. Suppose that $d \in \mathrm{cl}(\left\{ s(z) \; | \; z \in \vec{v} \right\} - \left\{ d \right\})$.  Let $\vec{x}' = (x_1, ..., x_{n-1})$ and $\vec{v} - \vec{x} = (v_0, ..., v_{k-1})$. By assumption $d \not\in \mathrm{cl}(\left\{ s(v) \; | \; v \in \vec{v} \cap \vec{x}' \right\}$ (remember that $ \vec{x} - \vec{v} \,\neq \emptyset$), thus there has to exists $h < k$ so that
	\[ d \in \mathrm{cl}(\left\{ s(v_i) \; | \; i \leq h \right\} \cup \left\{ s(v) \; | \; v \in \vec{v} \cap \vec{x}' \right\}) - \mathrm{cl}(\left\{ s(v_i) \; | \; i < h \right\} \cup \left\{ s(v) \; | \; v \in \vec{v} \cap \vec{x}' \right\}).  \]
	But then by the Exchange Principle we have that
		\[ s(v_h) \in \mathrm{cl}(\left\{ s(v_i) \; | \; i < h \right\} \cup \left\{ s(v) \; | \; v \in \vec{v} \cap \vec{x}' \right\}) \cup \left\{ d \right\}),  \]
contradicting what we have observed in {\em Case A}.

\smallskip

\noindent	{\em Case C.} $v \in \vec{v} \cap \vec{x}$ and $v \neq x_0$. If $x_0 \not\in \vec{v}$, then there is nothing to show because we are as in {\em Case 1}. Suppose then that $x_0 \in \vec{v}$ and $s(v) \in \mathrm{cl}(\left\{ s(z) \; | \; z \in \vec{v} \right\} - \left\{ s(v) \right\})$. As noticed, $s(v) \not\in \mathrm{cl}(\left\{ s(z) \; | \; z \in \vec{v} - x_0 \right\} - \left\{ s(v) \right\})$, thus by the Exchange Principle we have that
	\[ d \in \mathrm{cl}(\left\{ s(z) \; | \; z \in \vec{v} \right\} - \left\{ d \right\}), \]
contradicting what we have observed in {\em Case B}.
\smallskip	
	
\noindent	{\bf Case 2.} $\vec{x} \subseteq \vec{v}$.
This case is not possible. Suppose indeed it is, then by rule ($c_2.$) we can assume that $\vec{v} = \vec{x} \vec{v}'$ with $\vec{v}' \subseteq \mathrm{Var} - \vec{x}$. Thus by rule ($b_2.$) we have that $\Sigma \vdash \bota (\vec{x})$, contrary to our assumption.
\smallskip	
	
\noindent This concludes the proof of the theorem.
	
\end{proof}

\index{Pregeometry Atomic!Independence Logic} {\em Pregeometry Atomic Independence Logic} is defined as follows.
The syntax and deductive system of this logic are the same as those of Atomic Independence Logic (see Section \ref{abs_sys}).

Given a pregeometry $(M, \mathrm{cl})$ we   obtain an independence structure  $$\S=\S_{(M, \mathrm{cl})}=(M,\bot_\S)$$ by defining for finite subsets $\bx$ and $\by$ of $M$: 
$$\bx\ \bot_\S\ \by\iff \bx \displayclindep[\emptyset] \by. $$  We call $\S$ a {\em  pregeometry independence structure}. 
Semantics is defined by  $$\S \models_s \vx\boto \vy\mbox{ iff } s(\vx)\ \bot_\S\ s(\vy).$$

%
%

	\begin{theorem}\label{compl_pregeom_atom_indep_logic} 
Suppose $\C$ is a class of  independence structures including a structure $\S_{(M, \mathrm{cl})}$ with $(M, \mathrm{cl})$ satisfying (P1)-(P3).
	 Then the Completeness Question for $\C$ has a positive answer i.e. if $\Sigma$ is a set  of  independence atoms and $\phi$ is an  independence atom, then $$\Sigma\vdash_{\mathcal{I}}\phi\iff\Sigma\models_{\C} \phi.$$
\end{theorem}

\begin{proof} Soundness follows from Lemma \ref{prop_clindep}. Regarding completeness, let $\Sigma$ be a set of atoms and suppose that $\Sigma \nvdash \vec{x} \boto \vec{y}$. Notice that if this is the case then $\vec{x} \neq \emptyset$ and $\vec{y} \neq \emptyset$. Indeed if $\vec{y} = \emptyset$ then $\Sigma \vdash \vec{x} \boto \vec{y}$ because by rule ($a_3.$) $\vdash \vec{x} \boto \emptyset$. Analogously if $\vec{x} = \emptyset$ then $\Sigma \vdash \vec{x} \boto \vec{y}$ because by rule ($a_3.$) $\vdash \vec{y} \boto \emptyset$ and so by rule ($b_3.$) $\vdash \emptyset \boto \vec{y}$. 
	Furthermore we can assume that $\vec{x} \boto \vec{y}$ is minimal, in the sense that if $\vec{x}' \subseteq \vec{x}$, $\vec{y}' \subseteq \vec{y}$ and $\vec{x}' \neq \vec{x}$ or $\vec{y}' \neq \vec{y}$, then $\Sigma \vdash \vec{x}' \boto \vec{y}'$. This is for two reasons. 
	\begin{enumerate}[i)]
		\item If $\vec{x} \boto \vec{y}$ is not minimal we can always find a minimal atom $\vec{x}^* \boto \vec{y}^*$ such that $\Sigma \nvdash \vec{x}^* \boto \vec{y}^*$, $\vec{x}^* \subseteq \vec{x}$ and $\vec{y}^* \subseteq \vec{y}$ --- just keep deleting elements of $\vec{x}$ and $\vec{y}$ until you obtain the desired property or until both $\vec{x}^*$ and $\vec{y}^*$ are singletons, in which case, due to the trivial independence rule ($a_3.$), $\vec{x}^* \boto \vec{y}^*$ is a minimal statement. 
		\item For any $\vec{x}' \subseteq \vec{x}$, $\vec{y}' \subseteq \vec{y}$ and assignment $s$ we have that if $\S\not\models_s \vec{x}' \boto \vec{y}'$ then $\S\not\models_s \vec{x} \boto \vec{y}$ --- this follows easily by Monotonicity.

\end{enumerate}

\smallskip 

\noindent	Let $V = \left\{ v \in \mathrm{Var} \; | \; \Sigma \vdash v \boto v \right\}$, $W = \mathrm{Var} - V$, $\vec{x} \cap W = \vec{x}'$ and $\vec{y} \cap W = \vec{y}'$.
\smallskip

\noindent \begin{claim}\label{claim_for_proof} If $\Sigma \vdash \vec{x}' \boto \vec{y}'$, then $\Sigma \vdash \vec{x} \boto \vec{y}$.

\end{claim}

	\begin{claimproof} By induction on $t = |\vec{x} - \vec{x}'|$, we show that if $\Sigma \vdash \vec{x}' \boto \vec{y}'$, then $\Sigma \vdash \vec{x} \boto \vec{y}'$. By rule ($b_3.$) this suffices. Let $\vec{x} - \vec{x}'= (x_{k_0}, ..., x_{k_{t-1}})$ and $s \leq t$, then by induction assumption and  rules ($e_3.$), ($b_3.$) and ($d_3.$) we have that

	\[ \begin{array}{rcl}
	  \Sigma \vdash \vec{y}' \boto \vec{x}'  x_{k_0} \cdots x_{k_{s-2}} & \text{ and } & \Sigma \vdash \vec{y}'  \vec{x}'  x_{k_0} \cdots x_{k_{s-2}} \boto x_{k_{s-1}}      \\
													  	&  \Downarrow & \\
	&	\Sigma \vdash \vec{y}' \boto \vec{x}'   x_{k_0} \cdots x_{k_{s-1}}      &  
\end{array}
	 \]
and hence by rule ($f_3.$) and ($b_3.$) we have that $\Sigma \vdash \vec{x} \boto \vec{y}'$.

\end{claimproof}

\smallskip 

\noindent The claim above shows that if $\vec{x} \boto \vec{y}$ is minimal, then for every $z \in \vec{x}\vec{y}$ we have that $\Sigma \nvdash z \boto z$. Furthermore, because of rule ($g_3.$) we can assume that $\vec{x}$ and $\vec{y}$ are injective. This will be relevant in the following. We now make a case distinction. 
\smallskip 

\noindent \textbf{Case 1.} There exists $z \in \vec{x} \cap \vec{y}$. Remember that by assumption there exists $\S_{(M, \mathrm{cl})} \in \mathbf{C}$ such that $(M, \mathrm{cl})$ has the properties (P1), (P2) and (P3). Let $s$ be the following assignment:	
\begin{enumerate}[i)]
		\item $s(v) = e$ for every $v \in \mathrm{Var} - z$,
		\item $s(z) = d$.
\end{enumerate}
where $e \in \mathrm{cl}(\emptyset)$ and $d \not\in \mathrm{cl}(\emptyset)$, i.e. $d \not\!\clindep[\emptyset] d$. 

\smallskip

\noindent Obviously $\S \not\models_s \vec{x} \boto \vec{y}$, in fact $\S \not\models_s z \boto z$. Furthermore, for every $\vec{v} \boto \vec{w} \in \Sigma$, we have that $\S\models_s \vec{v} \boto \vec{w}$. Let indeed $\vec{v} \boto \vec{w} \in \Sigma$, then $z \notin \vec{v} \cap \vec{w}$, because otherwise, by rule ($c_3.$), we would have that $\Sigma \vdash z \boto z$, contrary to the minimality of $\vec{x} \boto \vec{y}$. Hence $\S\models_s \vec{v} \boto \vec{w}$, because by the choice of $e$ we have that $e \clindep[\emptyset] ed$.

\smallskip

\noindent	\textbf{Case 2.} $\vec{x} \cap \vec{y} = \emptyset$.

\noindent Let $\vec{x} = (x_0, ..., x_{n-1})$, $\vec{y} = (y_0, ..., y_{m-1})$ and $k = (n-1) + m$. Let $(w_i \; | \; i < k)$ be an injective enumeration of $\vec{x}  \vec{y} - x_0$ with $w_i = x_{i+1}$ for $i \leq n-2$ and $w_{i+{(n-1)}} = y_{i}$ for $i \leq m-1$. Also in this case, remember that by assumption there exists  $\S_{(M, \mathrm{cl})} \in \mathbf{C}$ such that $(M, \mathrm{cl})$ has the  properties (P1), (P2) and (P3). Let then $e \in \mathrm{cl}(\emptyset)$ and $D_0 = \left\{ a_i \; | \; i < k \right\} \subseteq M$ an independent set.
	Let $s$ be the following assignment:
	\begin{enumerate}[i)]
		\item $s(v) = e$ for every $v \in \mathrm{Var} - \vec{x}  \vec{y}$,
		\item $s(w_i) = a_i $ for every $i < k$,
		\item $s(x_{0}) = d$,
\end{enumerate}
where $d$ witnesses property (P2) with respect to the set $D_0$, i.e. $d \in \mathrm{cl}(D_0)$ and $d \not\in \bigcup_{D \subsetneq D_0} \mathrm{cl}(D)$.

\smallskip

\noindent	We claim that $\S\not\models_s \vec{x} \boto \vec{y}$. By the choice of $d$, we have that $d \not\!\clindep[\emptyset] a_0 \cdots a_{k-1}$. Suppose that $d a_0 \cdots a_{n-2} \clindep[\emptyset] a_{n-1} \cdots a_{k-1}$. Again by the choice of $d$, we have $d \clindep[\emptyset] a_0 \cdots a_{n-2}$, so by Exchange we have $d \clindep[\emptyset] a_0 \cdots a_{k-1}$, a contradiction. Thus, $d a_0 \cdots a_{n-2} \not\!\clindep[\emptyset] a_{n-1} \cdots a_{k-1}$ and hence $s(\vec{x}) \not\!\clindep[\emptyset] s(\vec{y})$.
\smallskip

\noindent	Let now $\vec{v} \boto \vec{w} \in \Sigma$, we want to show that $\S\models_s \vec{v} \boto \vec{w}$. 
Let $\vec{v} \cap \vec{x}  \vec{y} = \vec{v}'$ and $\vec{w} \cap \vec{x}  \vec{y} = \vec{w}'$.
	Notice that 
	\[ s(\vec{v}) \displayclindep[\emptyset] s(\vec{w}) \text{ if and only if } s(\vec{v}') \displayclindep[\emptyset] s(\vec{w}'). \]
Left to right holds in general. As for the other direction, suppose that $s(\vec{v}') \clindep[\emptyset] s(\vec{w}')$. If $u \in \vec{v}  \vec{w} - \vec{v}'  \vec{w}'$, then $s(u) = e$. Thus

	 \[ \begin{array}{rcl}
	 &\;\;\; s(\vec{v}') \clindep[\emptyset] s(\vec{w}')  \text{ and } s(\vec{v}')  s(\vec{w}') \clindep[\emptyset] e         & [\text{By Anti-Reflexivity}]  \\
	        											  &  \Downarrow & \\
	 &\;\;\;\;\, s(\vec{v}')  \clindep[\emptyset] s(\vec{w}')  e          &  \\
	  													  &  \Downarrow & \\
	  													  &  \;\,\,s(\vec{v}')  \clindep[\emptyset] s(\vec{w}). &
\end{array}
	 \]
So
	\[ \begin{array}{rcl}
 	&\;\;\; s(\vec{w}) \clindep[\emptyset] s(\vec{v}')  \text{ and } s(\vec{w})  s(\vec{v}') \clindep[\emptyset] e          & 			   [\text{By Anti-Reflexivity}]  \\
        											  &  \Downarrow & \\
 	&\;\;\;\; s(\vec{w})  \clindep[\emptyset] s(\vec{v}')  e          &  \\
  													  &  \Downarrow & \\
  													  &  \;\,s(\vec{w})  \clindep[\emptyset] s(\vec{v}). &
\end{array}
 \]

\smallskip

\noindent Notice that $\vec{v}' \cap \vec{w}' = \emptyset$. Indeed, suppose that there exists $z \in \vec{v}' \cap \vec{w}'$, then, by rule ($c_3.$), we have that $\Sigma \vdash z \boto z$, contrary to the minimality of $\vec{x} \boto \vec{y}$. We make another case distinction.
\smallskip

\noindent	\textbf{Subcase 1.} $x_{0} \notin \vec{v}' \vec{w}'$. 
	 As noticed, $\vec{v}' \cap \vec{w}' = \emptyset$, and so, by properties of our assignment $s(\vec{v}') \cap s(\vec{w}') = \emptyset$. Thus, by Lemma~\ref{lemma_indep_seq}, it follows that $s(\vec{v}') \clindep[\emptyset] s(\vec{w}')$.
\smallskip

\noindent	\textbf{Subcase 2.} $x_{0} \in \vec{v}' \vec{w}'$. There are two subcases.
\smallskip

\noindent	\textbf{Subcase 2.1.} $(\vec{x}  \vec{y}) - (\vec{v}'  \vec{w}') \neq \emptyset$. 
		Let $\vec{a} = s(\vec{v}') - d $ and $\vec{b} = s(\vec{w}') - d$. By assumption we have that $(\vec{x} - x_0) \cup \vec{y} \not\subseteq \vec{v}'  \vec{w}'$ and so $s((\vec{x} - x_0) \cup \vec{y}) \not\subseteq \vec{a}  \vec{b}$. Thus, by the choice of $d$, we have that $\vec{a} \vec{b} \clindep[\emptyset] d$. Suppose now that $x_{0} \in \vec{w}'$, the other case is symmetrical. By properties of our assignment $\vec{a} \cap \vec{b} = \emptyset$, hence by Lemma~\ref{lemma_indep_seq}, we have that $\vec{a} \clindep[\emptyset] \vec{b}$. Thus, by Exchange, $\vec{a} \clindep[\emptyset] \vec{b} d$. Hence, permuting the elements in $\vec{b} d$, we conclude that $s(\vec{v}) \clindep[\emptyset] s(\vec{w})$.
\smallskip

\noindent	\textbf{Subcase 2.2.} $\vec{x}  \vec{y} \subseteq \vec{v}'  \vec{w}'$. 
		This case is not possible. By rule ($f_3.$) and ($c_3.$) we can assume that $\vec{v} = \vec{v}'  \vec{u}$ and $\vec{w} = \vec{w}'  \vec{u}'$ with $\vec{u}  \vec{u}' \subseteq \mathrm{Var} - \vec{v}'  \vec{w}'$. Furthermore because $\vec{x}  \vec{y} \subseteq \vec{v}'  \vec{w}'$ again by rule ($f_3.$) we can assume that $\vec{v}' = \vec{x}'  \vec{y}'  \vec{z}'$ and $\vec{w}' = \vec{x}''  \vec{y}''  \vec{z}''$ with $\vec{x}'  \vec{x}'' = \vec{x}$, $\vec{y}'  \vec{y}'' = \vec{y}$ and $\vec{z}' \vec{z}'' \subseteq \mathrm{Var} - \vec{x}  \vec{y}$. Hence $\vec{v} = \vec{x}'  \vec{y}'  \vec{z}'  \vec{u}$ and $\vec{w} = \vec{x}''  \vec{y}''  \vec{z}''  \vec{u}'$.
		By hypothesis we have that $\vec{v} \boto \vec{w} \in \Sigma$ so by rules ($c_3.$) and ($b_3.$) we can conclude that $\Sigma \vdash \vec{x}'  \vec{y}' \boto \vec{x}''  \vec{y}''$.
		If $\vec{x}' = \vec{x}$ and $\vec{y}'' = \vec{y}$, then $\Sigma \vdash \vec{x} \boto \vec{y}$ because as we noticed $\vec{v}' \cap \vec{w}' = \emptyset$, a contradiction. Analogously if $\vec{x}'' = \vec{x}$ and $\vec{y}' = \vec{y}$, then $\Sigma \vdash \vec{y} \boto \vec{x}$. Thus by rule ($b_3.$) $\Sigma \vdash \vec{x} \boto \vec{y}$, a contradiction. There are then four cases:
		\begin{enumerate}[i)]
			\item $\vec{x}' \neq \vec{x}$ and $\vec{x}'' \neq \vec{x}$;
			\item $\vec{y}' \neq \vec{y}$ and $\vec{x}'' \neq \vec{x}$;
			\item $\vec{y}' \neq \vec{y}$ and $\vec{y}'' \neq \vec{y}$;
			\item $\vec{x}' \neq \vec{x}$ and $\vec{y}'' \neq \vec{y}$.

	\end{enumerate}
		Suppose that either i) or ii) holds. If this is the case, then $\Sigma \vdash \vec{x}' \boto \vec{y}'$ because by hypothesis $\vec{x} \boto \vec{y}$ is minimal. So $\Sigma \vdash \vec{x}' \boto \vec{y}'  \vec{x}''  \vec{y}''$, because by rule ($d_3.$) 
		\[ \Sigma \vdash \vec{x}' \boto \vec{y}' \text{ and } \Sigma \vdash \vec{x}'  \vec{y}' \boto \vec{x}''  \vec{y}''\ \Rightarrow \Sigma \vdash \vec{x}' \boto \vec{y}'  \vec{x}''  \vec{y}''.\]
		Hence by rule ($e_3.$) $\Sigma \vdash \vec{x}' \boto \vec{x}''  \vec{y}$ and then by rule ($b_3.$)  $\Sigma \vdash \vec{x}''  \vec{y} \boto \vec{x}'$. So by rule ($e_3.$) $\Sigma \vdash \vec{y}  \vec{x}'' \boto \vec{x}'$.
		We are under the assumption that $\vec{x}'' \neq \vec{x}$ thus again by minimality of $\vec{x} \boto \vec{y}$ we have that $\Sigma \vdash \vec{x}'' \boto \vec{y}$ and so by rule ($b_3.$) we conclude that $\Sigma \vdash \vec{y} \boto \vec{x}''$. Hence $\Sigma \vdash \vec{y} \boto \vec{x}''  \vec{x}'$, because by rule ($d_3.$) 
		\[ \Sigma \vdash \vec{y} \boto \vec{x}'' \text{ and } \Sigma \vdash \vec{y}  \vec{x}'' \boto \vec{x}' \Rightarrow \Sigma \vdash \vec{y} \boto \vec{x}''  \vec{x}'.\]
		Then finally by rules ($e_3.$) and ($b_3.$) we can conclude that $\Sigma \vdash \vec{x} \boto \vec{y}$, a contradiction.
		The case in which either iii) or iv) holds is symmetrical.
\smallskip

\noindent This concludes the proof of the theorem.

\end{proof}

%

	Notice that the notions of independence introduced in Section \ref{pregeometries} are conditional, i.e. they talk about independence {\em over} a set of parameters. It is then possible to formulate a conditional version of the system Pregeometry Atomic Independence Logic based on the syntax and deductive system of Atomic Conditional Independence Logic (see Section \ref{abs_sys}). From Lemma \ref{prop_clindep} it follows that this system is sound. We do not know whether the given axioms are complete, though. This is an open problem at the moment. By soft arguments (reduction to first-order logic by means of extra predicates) one can argue that there is {\em some} effective axiomatization, we just do not have an explicit one.
	
\section{Dependence and Independence in $\omega$-Stable Theories}\label{daiist}

	In this section we generalize our work to cover the case of independence at play in $\omega$-stable theories. Much stronger results are possible, see \cite{pao}.

\subsection{Forking and Strongly Minimal Sets}
	
	Let $\mathcal{M}$ be a first-order structure and $A \subseteq M$, we denote by $\typesp{M}{n}{A}$ the space of complete $n$-types over $A$.
	
	\begin{definition} Let $T$ be a complete first-order theory in a countable language with infinite models and let $\kappa$ be an infinite cardinal. We say that $T$ is \emph{$\kappa$-stable} if whenever $\mathcal{M} \models T$, $A \subseteq M$ and $|A| = \kappa$, then $|\typesp{M}{n}{A}| = \kappa$.

\end{definition} \index{theory!$\kappa$-stable}
\noindent	Let $T$ be an $\omega$-stable theory, $\mathcal{M} \models T$, $A \subseteq M$ and $p \in \typesp{M}{n}{A}$, we denote by $\mathrm{RM}(p)$ the {\em Morley rank} of $p$. Furthermore, we denote by $\mathfrak{M}$ the monster model of $T$. For details see \cite[Chapter 6]{marker}.

\begin{definition} Let $\mathcal{M} \models T$, $A \subseteq B \subseteq M$, $p \in \typesp{M}{n}{A}$, $q \in \typesp{M}{n}{B}$ and $p \subseteq q$. If $\mathrm{RM}(q) < \mathrm{RM}(p)$, we say that $q$ is a \emph{forking extension} of $p$ and that $q$ forks over $A$. If $\mathrm{RM}(q) = \mathrm{RM}(p)$, we say that $q$ is a \emph{non-forking extension} of $p$.
	
\end{definition} \index{forking}

\begin{definition} Let $\mathfrak{M}$ be the monster model of $T$ and $A, B, C \subseteq \mathfrak{M}$. We say that $A$ is \emph{independent} from $B$ over $C$ if for every $\vec{a} \in A^{< \omega}$ we have that $\mathrm{tp}(\vec{a}/C \cup B)$ is a non-forking extension of $\mathrm{tp}(\vec{a}/C)$. In this case we write $A \forkindep[C] B$.
	
\end{definition}

	This notion of independence has all the properties that we observed in the case of pregeometric independence in Section \ref{pregeometries}, as we see in the next well-known lemma (see e.g. \cite[Lemma~6.3.16-21]{marker}).

\begin{lemma}\label{prop_frkindep} Let $\mathfrak{M}$ be the monster model of $T$ and $A, B, C, D \subseteq \mathfrak{M}$. Then
\begin{enumerate}[i)]
\item $A \forkindep[A] B$ \; [Existence];
\item if $A \forkindep[C] B$ and $D \subseteq B$, then $A \forkindep[C] D$ \; [Monotonicity];
\item $A \forkindep[C] B \cup D$ if and only if $A \forkindep[C] B$ and $A \forkindep[C \cup B] D$ \; [Transitivity];
\item $A \forkindep[C] B$ if and only if $A \forkindep[C] B_0$ for all finite $B_0 \subseteq B$ \; [Finite Character];
\item if $A \forkindep[C] B$, then $B \forkindep[C] A$ \; [Symmetry];
\item if $A \forkindep[D] B$ and $A \cup B \forkindep[D] C$, then $A \forkindep[D] B \cup C$ \; [Exchange];
\item if $A \forkindep[B] A$, then $A \forkindep[B] E$ for any $E \subseteq M$.

\end{enumerate}

\end{lemma}

We conclude this section stating a fundamental theorem about Morley rank in strongly minimal sets, which will play a crucial role in the following. We first give some definitions.
	
\begin{definition} Let $\mathcal{M}$ be an $L$-structure and let $D \subseteq M^n$ be an infinite definable set. We say that $D$ is \emph{minimal} in $\mathcal{M}$ if for every definable $Y \subseteq D$ either $Y$ is finite or $D - Y$ is finite. If $\phi(\vec{v}, \vec{a})$ is the formula that defines $D$, then we also say that $\phi(\vec{v}, \vec{a})$ is minimal. 
	We say that $D$ and $\phi$ are \emph{strongly minimal} if $\phi$ is minimal in any elementary extension $\mathcal{N}$ of $\mathcal{M}$. 
	We say that a theory $T$ is strongly minimal if the formula $v = v$ is strongly minimal (i.e. if $\mathcal{M} \models T$, then $M$ is strongly minimal).
	
\end{definition}

\noindent	If $\mathcal{M}$ is a model of an $\omega$-stable theory and $F = \phi(\mathcal{M}, \vec{c})$ is a strongly minimal set we can define a pregeometry $(F, \mathrm{cl}_{\vec{c}})$ by defining $\mathrm{cl}_{\vec{c}}(X) = \mathrm{acl}(\vec{c} \cup X) \cap F$. When $F = \phi(\mathcal{M}, \emptyset)$, we denote the pregeometry $(F, \mathrm{cl}_{\emptyset})$ simply as $(F, \mathrm{cl})$. We are now in the position to state the announced theorem, for a proof see \cite[Theorem~6.2.19]{marker} or \cite[Section 1.5]{pillay}.

	\begin{theorem}\label{rank_eq_dim} Let $T$ be an $\omega$-stable theory, $F = \phi(\mathcal{M}, \vec{c})$ a strongly minimal set, $\vec{a} \in F$ and $A \subseteq F$. Then $\mathrm{RM}(\mathrm{tp}(\vec{a}/A \cup \vec{c})) = \mathrm{dim}(\vec{a}/A)$, where $\mathrm{dim}(\vec{a}/A)$ is  computed in the pregeometry $(F, \mathrm{cl}_{\vec{c}})$.
		
\end{theorem}

\subsection{$\omega$-Stable Atomic Independence Logic}

\index{$\omega$-Stable Atomic!Independence Logic} We define {\em $\omega$-Stable Atomic Independence Logic} as follows. 
The syntax and deductive system of this logic are the same as those of Atomic Independence Logic (see Section \ref{abs_sys}).
%
%
	Let $T$ be a first-order $\omega$-stable theory. For every $\mathcal{M} \models T$ we obtain an independence structure $$\S=\S_{(\mathcal{M}, \mbox{\scriptsize $\forkindep$}\!\hspace{-0.5pt}\hspace{-0.5pt})}=(M,\bot_\S)$$ by defining for finite subsets $\bx$ and $\by$ of $M$: 
$$\bx\ \bot_\S\ \by\iff \bx \displayforkindep[\emptyset] \by. $$  We call $\S$ a {\em  forking independence structure}. 
Semantics is defined by  $$\S \models_s \vx\boto \vy\mbox{ iff } s(\vx)\ \bot_\S\ s(\vy).$$

%
%

	\begin{theorem}\label{saind} 
Suppose $\C$ is a class of   independence structures including a forking independence structure $\S_{(\mathcal{M}, \mbox{\scriptsize $\forkindep$}\!\hspace{-0.5pt}\hspace{-0.5pt})}$ such that there exists a strongly minimal set $F = \phi(\mathcal{M}, \emptyset)$ such that $(F, \mathrm{cl})$ has properties (P1), (P2) and (P3).
	 Then the Completeness Question for $\C$ has a positive answer i.e. if $\Sigma$ is a set  of  independence atoms and $\phi$ is an  independence atom, then $$\Sigma\vdash_{\mathcal{I}}\phi\iff\Sigma\models_{\C} \phi.$$
\end{theorem}

%
%
%

	\begin{proof} Soundness follows from Lemma \ref{prop_frkindep}. Regarding completeness, notice that by Lemma \ref{rank_eq_dim} independence inside a strongly minimal set $F$ coincides with non-forking. Thus, the proof of Theorem \ref{compl_pregeom_atom_indep_logic} goes through, i.e. one can build the assignment $s$ for the pregeometry $(F, \mathrm{cl})$ as we did in Theorem \ref{compl_pregeom_atom_indep_logic}.

\end{proof}

	As in the case of pregeometries, it is possible to formulate a conditional version of the system just described based on the syntax and deductive system of Atomic Conditional Independence Logic. By Lemma \ref{prop_frkindep}, we know that the system is sound but we do not know if it is complete. For particular theories $T$ we can make a further conclusion\footnote{We are indebted to Tapani Hyttinen for pointing this out.}. Suppose $T$ is the theory of vector spaces over a fixed finite field. Then $T$ is decidable, so in this case we get the decidability of the relation $\Sigma \models \vec{x} \botc{\vec{z}} \vec{y}$ for finite $\Sigma$.

\subsection{$\omega$-Stable Atomic Dependence Logic}

	As known \cite{GV}, in dependence logic the dependence atom is expressible in terms of the conditional independence atom. Indeed, for any team $X$ we have 
	\[ \D_X \models_s \dep(\vec{x}, \vec{y}) \text{ if and only if } \I_X \models_s \vec{y} \botc{\vec{x}} \vec{y}. \]
	
\index{$\omega$-Stable Atomic!Dependence Logic} We now define {\em $\omega$-Stable Atomic Dependence Logic} 
  as follows. The syntax and deductive system of this system are the same as those of Atomic Dependence Logic (see Section \ref{abs_sys}).
%
%
	Let $T$ be an $\omega$-stable first-order  theory.
%
For every $\mathcal{M} \models T$ we obtain a dependence structure $$\D=\D_{(\mathcal{M}, \mbox{\scriptsize $\forkindep$}\!\hspace{-0.5pt}\hspace{-0.5pt})}=(M,\rr_\D)$$ by defining for finite subsets $\bx$ and $\by$ of $M$: 
$$\bx\rr\by\mbox{ if and only if } \by \displayforkindep[\bx] \by. $$  We call $\D$ a {\em  forking dependence structure}. 
Semantics of dependence atoms $\dep(\vx,\vy)$ is defined by  $$\D \models_s \dep(\vx,\vy)\mbox{ if and only if } s(\vx)\rr_\D s(\vy).$$
	
	\begin{lemma}\label{char_stab_dep_with_acl_dep} Let $\mathcal{M} \models T$ and $\vec{a}, \vec{b} \in M$, then 
		\[ \vec{b} \displayforkindep[\vec{a}] \vec{b} \text{ if and only if } \forall b \in \vec{b} \; \; b \in \mathrm{acl}(\vec{a}).\]
		
\end{lemma}

\begin{proof} \[ \begin{array}{rcl}
	\vec{b} \forkindep[\vec{a}] \vec{b}  & \Longleftrightarrow &   \mathrm{RM}(\mathrm{tp}(\vec{b}/\vec{a} \cup \vec{b})) = \mathrm{RM}(\mathrm{tp}(\vec{b}/\vec{a}))\\
												      				  &  \Longleftrightarrow &  \mathrm{RM}(\mathrm{tp}(\vec{b}/\vec{a})) = 0 \\
												      				  &  \Longleftrightarrow &  \exists \phi(\vec{v}) \in \mathrm{tp}(\vec{b}/\vec{a}) \text{ s.t. } |\phi(\mathcal{M})| < \infty  \\	
													  				  &  \Longleftrightarrow &  \vec{b} \in \mathrm{acl}(\vec{a})) \\	
												      				  &  \Longleftrightarrow &  \forall b \in \vec{b} \;\; b \in \mathrm{acl}(\vec{a}).													

\end{array}
 \]
	
\end{proof}

%
	\begin{theorem} 
Suppose $\C$ is a class of dependence structures containing a forking dependence structure $\D_{(\mathcal{M}, \mbox{\scriptsize $\forkindep$}\!\hspace{-0.5pt}\hspace{-0.5pt})}$ such that the closure system $(M, \mathrm{acl})$ satisfies $\emptyset\ne\mathrm{acl}(\emptyset) \neq M$.
	 Then the Completeness Question for $\C$ has a positive answer i.e. if $\Sigma$ is a set  of  dependence atoms and $\phi$ is a dependence atom, then $$\Sigma\vdash_{\mathcal{D}}\phi\iff\Sigma\models_{\C} \phi.$$
\end{theorem}	
		
	\begin{proof} Because of Lemma~\ref{char_stab_dep_with_acl_dep}, this reduces to Theorem~\ref{compl_clo_op_dep_logic}.

\end{proof}

\subsection{Forking in Vector Spaces and Algebraically Closed Fields}\label{vsacf}

	In this last section we show that the conditions imposed on $\omega$-stable $T$ in the definition of the system $\omega$-Stable Atomic Independence Logic are satisfied by the theory of infinite vector spaces over a countable field and the theory of algebraically closed fields of fixed characteristic. These facts are very well-known among model theorists, we include them for completeness of exposition.

We denote by $\mathrm{VS}^{\mathrm{inf}}_{\mathbb{K}}$ the theory of infinite vector spaces over a fixed field $\mathbb{K}$.
	
	\begin{proposition}\label{inf_vs_str_min} The theory $\mathrm{VS}^{\mathrm{inf}}_{\mathbb{K}}$ is strongly minimal.

\end{proposition} 



%
\noindent	Let $\mathbb{K}$ be a countable field.

	\begin{proposition} 
		\begin{enumerate}[i)]
			\item The theory $\mathrm{VS}^{\mathrm{inf}}_{\mathbb{K}}$ is $\omega$-stable.
			\item The theory $\mathrm{VS}^{\mathrm{inf}}_{\mathbb{K}}$ has a strongly minimal set $F = \phi(\mathfrak{M}, \emptyset)$ such that $(F, \mathrm{cl})$ has properties (P1), (P2) and (P3).
		\end{enumerate}

\end{proposition} 
	
		

		

		

	For $p = 0$ or a prime number we denote by $\mathrm{ACF}_p$ the theory of algebraically closed fields of characteristic $p$.

	\begin{proposition}
		\begin{enumerate}[i)]
		\item The theory $\mathrm{ACF}_p$ is strongly minimal.
		\item The theory $\mathrm{ACF}_p$ is $\omega$-stable.
		\item The theory $\mathrm{ACF}_p$ has a strongly minimal set $F = \phi(\mathfrak{M}, \emptyset)$ such that $(F, \mathrm{cl})$ has properties (P1), (P2) and (P3).
		
\end{enumerate}
		
\end{proposition} 

	
	
	


\section{Conclusion}

We introduced several forms of dependence and independence logics:

\begin{itemize}
\item  Dependence Logic with atoms $\Dep{x}{y}$
\item  Absolute Independence Logic with atoms $\bota (\vec{x})$
\item  Independence Logic with atoms $\vec{x} \boto \vec{y}$
\item  Conditional Independence Logic with atoms $\vec{x} \botc{\vec{z}} \vec{y}$

\end{itemize}
that we studied on the {\em atomic} level, where we already have non-trivial questions about axiomatizability and decidability. We gave these concepts meaning in different contexts such as:
\begin{itemize}
\item Team semantics (i.e. databases)
\item Closure Operators 

\item Pregeometries 
\item $\omega$-Stable first-order theories.

\end{itemize}In several important cases we found that the common axioms of independence, going back to Whitney \cite{whitney} and van der Waerden \cite{waerden}, are complete on the atomic level. As pointed out in \cite{KLV}, these are the same axioms that govern central concepts of independence in database theory. Furthermore, it can be argued that these axioms govern the concepts of dependence and independence in a whole body of areas of science and the humanities. Thus the concepts of dependence and independence enjoy a  remarkable degree of robustness, reminiscent of their central role in mathematics and its applications. The situation of this paper, where variables are interpreted as elements of a structure with a pregeometry, is naturally much richer than the more general environment of database theory, or team semantics, where variables are interpreted as vectors with no field structure on the coefficients. However, our results demonstrate that it is possible to consider algebra, and by the same token model theory of stable theories, in the same framework with databases and other more general structures where dependence and independence concepts make sense. It is possible that database theory benefits from such a common framework, but also that algebra and model theory benefit from this connection. At least the concepts of dependence and independence have, by virtue of their axioms, meaning that crosses over the territory from database theory all the way to algebra and model theory. 


	



\def\cprime{$'$} \def\Dbar{\leavevmode\lower.6ex\hbox to 0pt{\hskip-.23ex
  \accent"16\hss}D} \def\cprime{$'$}

\end{document}